\documentclass[reqno,12pt]{amsart}

\usepackage{times}
\usepackage{lmodern}
\usepackage{textcomp}
\usepackage{amsmath,cancel}
\usepackage{amsfonts,enumitem}
\usepackage{amssymb}
\usepackage{color}
\usepackage{mathrsfs}
\usepackage[dvipsnames]{xcolor}
\usepackage{esint}
\usepackage[colorlinks=true]{hyperref}
\usepackage{cleveref}
\newtheorem{thm}{Theorem}[section]
\newtheorem{lem}[thm]{Lemma}
\newtheorem{cor}[thm]{Corollary}
\newtheorem{prop}[thm]{Proposition}

\newtheorem{rem}[thm]{Remark}

\numberwithin{equation}{section}

\newcommand{\R}{\mathbb{R}}

\newcommand{\ve}{\varepsilon}
\newcommand{\rd}{\mathrm{d}}

\newcommand{\dhr}{\mathrel{\lhook\joinrel\relbar\kern-.8ex\joinrel\lhook\joinrel\rightarrow}}

\newcommand{\ccyan}[1]{{#1}}
\newcommand{\cblue}[1]{{#1}}

\begin{document}

\title[Chemotaxis Model with Indirect Signal Production]{A Sharp Mass Threshold for Boundedness in a Critical Quasilinear Chemotaxis Model with Indirect Signal Production: The Radially Symmetric Case}

\author{Philippe Lauren\c{c}ot}
\address{Laboratoire de Math\'ematiques (LAMA) UMR~5127, Universit\'e Savoie Mont Blanc, CNRS\\	F--73000 Chamb\'ery, France}
\email{philippe.laurencot@univ-smb.fr}
\author{Christoph Walker}
\address{Leibniz Universit\"at Hannover\\
Institut f\"ur Angewandte Mathematik\\
Welfengarten 1\\
30167 Hannover\\
Germany}
\email{walker@ifam.uni-hannover.de}
\date{\today}

\begin{abstract}
A critical quasilinear chemotaxis model with indirect signal production in space dimensions higher than two is studied under radial symmetry. Exploiting the variational structure of the system, the existence of a sharp mass threshold $M_c>0$ is established: all solutions with mass below $M_c$ are global and bounded, while those with mass above $M_c$ and sufficiently negative energy are global but blow up in infinite time. Moreover, the value of the critical mass $M_c$ is related to the best constant of a variant of a Hardy-Littlewood-Sobolev inequality.
\end{abstract}
%
\maketitle

\section{Introduction}\label{sec.1}

The J\"ager-Luckhaus version of the parabolic-elliptic Keller-Segel chemotaxis system with indirect signal production 
\begin{equation}\label{jl}
\begin{aligned}
	& \partial_t u = \mathrm{div}\big( \nabla u - u \nabla v \big), & \qquad (t,x)\in (0,\infty)\times\Omega, \\
	& 0 = \Delta v + w - \frac{1}{|\Omega|} \int_\Omega w\ \mathrm{d}x\,, \quad \int_\Omega v\ \mathrm{d}x = 0, & \qquad (t,x)\in (0,\infty)\times\Omega, \\
	& \partial_t w = u - w, & \qquad (t,x)\in (0,\infty)\times\Omega, \\
	& \nabla u\cdot \mathbf{n} = \nabla v\cdot \mathbf{n} = 0, & \qquad (t,x)\in (0,\infty)\times\partial\Omega, \\
	& (u,w)(0) = (u^0,w^0), & \qquad x\in\Omega,
\end{aligned}
\end{equation}
exhibits an intriguing dynamical behavior in a disk $\Omega$ of $\mathbb{R}^2$, uncovered in \cite{TW2017}. More precisely, while all non-negative radially symmetric classical solutions are global, there is an explicit sharp mass threshold $8\pi$ with the following property: if $\|u^0\|_1<8\pi$, then the corresponding solution $(u,v,w)$ to~\eqref{jl} is bounded uniformly with respect to time and space; that is,
\begin{equation*}
	\sup_{t\ge 0} \left( \|u(t)\|_\infty +  \|v(t)\|_\infty +  \|w(t)\|_\infty \right) < \infty\,. 
\end{equation*}
In contrast, given $M>8\pi$, there are non-negative radially symmetric initial data $(u^0,w^0)$  with $\|u^0\|_1=M$ for which the first component $u$ of the corresponding solution $(u,v,w)$ to~\eqref{jl} is unbounded, the $L_\infty$-norm of $u(t)$ growing at least at an exponential rate. It is further shown in \cite[Theorem~1.1]{ML2024} and \cite[Theorem~1.1]{DLY2026} that there is $M^*\in [8\pi,32\pi)$ such that, if $\|u_0\|_1>M^*$, then solutions emanating from sufficiently concentrated initial data are unbounded and collapse to a Dirac mass as $t\to\infty$. On the one hand, this property is clearly reminiscent of the dynamical behavior of the Smoluchowski-Poisson equation and the parabolic-elliptic Keller-Segel chemotaxis system, for which it is known that the mass threshold $8\pi$ separates global existence from the onset of finite time blowup for radially symmetric solutions in a two-dimensional disk, see \cite{Bi2020} and the references therein. On the other hand, the onset of unboundedness beyond a sharp mass threshold is not restricted to the radially symmetric setting in space dimension~$2$ and is also true in a general bounded domain of $\mathbb{R}^2$ and for related models including the parabolic-parabolic Keller-Segel system with indirect signal production, though with a different threshold value \cite{La2019}. In higher space dimensions $d\ge 3$, unbounded non-negative radially symmetric solutions are constructed whatever the value of $\|u_0\|_1$ in \cite[Proposition~1.1]{FLT2023}, and the occurrence of finite time blowup for sufficiently concentrated initial data is shown in \cite{JL2026}.

\medskip

A similar issue is investigated recently in \cite{FLT2023} for a quasilinear version of~\eqref{jl} in higher space dimension $d\ge 3$, which reads
\begin{subequations}\label{E}
\begin{align}
	&\partial_t u  = \mathrm{div}\big(m(1+u)^{m-1}\nabla u-u\nabla v\big)\,, \qquad &(t,x)\in (0,\infty)\times\Omega\,, \label{E1}\\
	&0=\Delta v+w-\langle w\rangle\,, \quad \langle v \rangle = 0\,, \qquad &(t,x)\in (0,\infty)\times\Omega\,, \label{E2}\\
	&\partial_t w=u-w\,, \qquad &(t,x)\in (0,\infty)\times\Omega\,, \label{E3} \\
	&\nabla u\cdot \mathbf{n}  = \nabla v\cdot \mathbf{n} = 0\,, \qquad &(t,x)\in (0,\infty)\times\partial\Omega\,, \label{E4} \\
	&(u,w)(0)  = (u^0,w^0)\,, \qquad &x\in \Omega\,, \label{E5}
\end{align}
\end{subequations}
where $\Omega$ is a bounded and smooth domain of $\mathbb{R}^d$, 
\begin{equation*}
    m=m_c:= \frac{2(d-1)}{d}\in (1,2)\,,
\end{equation*}
and 
\begin{equation*}
	\langle f\rangle:= \frac{1}{|\Omega|} \int_\Omega f(x)\,\mathrm{d} x\,,\quad f\in L_1(\Omega)\,.
\end{equation*}
The occurrence of a critical mass phenomenon is shown in \cite{FLT2023}. More precisely, on the one hand, the existence of a number $M_{FLT}^l>0$ is established such that solutions corresponding to any non-negative initial data with $\|u^0\|_1<M_{FLT}^l$ exist globally and remain bounded. On the other hand, when $\Omega$ is a ball, it is shown that, for each $M>M_{FLT}^u := 2 (2\pi)^{d/2} d^{d-1}/\Gamma(d/2)$, there are non-negative radially symmetric initial data with $\|u^0\|_1=M$  for which the solution blows up in infinite time. However, the value of $M_{FLT}^l$ is not identified and no information is provided for initial values with $M\in \big(M_{FLT}^l,M_{FLT}^u\big)$ if any. Let us recall that such a critical mass phenomenon only occurs for $m=m_c$. Indeed, for $m>m_c$, all solutions to~\eqref{E} are global while, for $m\in (1,m_c)$, there are solutions to~\eqref{E} blowing in finite time whatever the value of $\|u^0\|_1$, see \cite{FLT2023, JL2026}.

\medskip

The main purpose of this work is to exploit the variational structure of~\eqref{E} in order to identify \textit{precisely} the critical mass threshold $M_c>0$ separating boundedness and unboundedness of radially symmetric global solutions in a ball. In fact, as we shall see below, $M_c>0$ coincides with the critical mass for the quasilinear Patlak–Keller–Segel model identified in \cite{BCL2009}. Moreover, owing to the availability of an energy functional associated with~\eqref{E}, we show that \textit{any} solution with initial mass $\|u^0\|_1>M_c$ and sufficiently negative energy exhibits infinite time blowup.

\section{Main results}\label{sec.mr}

We first establish the global well-posedness of~\eqref{E} and set up some notation. For $p\in (1,\infty)$, we denote the domain of the Laplace operator in $L_p(\Omega)$ supplemented with homogeneous Neumann boundary conditions by $W_{p,N}^2(\Omega)$; that is, 
\begin{equation*}
	W_{p,N}^2(\Omega) := \{ f\in W_p^2(\Omega)\, :\, \nabla f\cdot \mathbf{n} = 0 \;\text{ on }\; \partial\Omega\}.
\end{equation*}
We also introduce the isomorphism
\begin{subequations}\label{opK}
\begin{equation}
	\mathcal{K}:\big\{f\in  L_p(\Omega)\,:\, \langle f\rangle=0\big\}\rightarrow\big\{f\in W_{p,N}^2(\Omega)\,:\, \langle f\rangle=0\big\}\,, \label{opKa}
\end{equation} 
defined by
\begin{equation}
	-\Delta \mathcal{K}[f]=f\ \text{ in }\ \Omega\,,\quad \nabla \mathcal{K}[f]\cdot \mathbf{n}=0 \ \text{ on }\ \partial\Omega\,,\quad \langle \mathcal{K}[f]\rangle=0\,, \label{opKb}
\end{equation}
\end{subequations}
for $f\in  L_p(\Omega)$ with $\langle f\rangle=0$. We write $\mathbb{B}_R$ for the ball in $\R^d$ of radius $R>0$ centered at the origin.

\begin{thm}\label{T1:Ex}
Let $p\in (2d,\infty)$ and consider an initial value
\begin{equation}\label{a1}
    (u^0,w^0)\in W_{p}^{1,+}(\Omega)  \times L_\infty^+(\Omega) \,.
\end{equation}
Then there is a unique couple $(u,w)$ of functions satisfying
\begin{subequations}
\begin{align}
&u \in  C^1\big((0,\infty),L_p(\Omega)\big) \cap  C\big((0,\infty),W_{p,N}^{2}(\Omega)\big)\cap C\big([0,\infty),W_{p}^{1}(\Omega) \big)\,,\label{regularityA}\\
& w\in C^1 \big([0,\infty),L_\infty(\Omega)\big)\,,\qquad v := \mathcal{K}[w-\langle w\rangle] \in C^1\big([0,\infty),W_{p,N}^{2}(\Omega)\big)\,,\label{regularityB}
\end{align}
such that $(u,v,w)$ is a global non-negative strong solution to~\eqref{E}.
\end{subequations}
In fact, the mapping 
\begin{equation*}
	\boldsymbol{\Psi}:(u^0,w^0)\mapsto (u,w)
\end{equation*} 
defines a global semiflow on $W_{p}^{1,+}(\Omega) \times L_\infty^+(\Omega)$.  
Moreover, 
\begin{itemize}
    \item  for $t\ge 0$, it holds that
    \begin{equation}
        \|u(t)\|_1 = \|u^0\|_1\,, \quad \|w(t)\|_1 \le \max\big\{ \|w^0\|_1 , \|u^0\|_1 \big\}\,; \label{L1}
    \end{equation}
	\item  if $u\in L_\infty((0,\infty),L_m(\Omega))$, then $(u,w)\in L_\infty((0,\infty)\times\Omega,\R^2)$;
	\item  if $\Omega=\mathbb{B}_R$  and $(u^0,w^0)$ are radially symmetric, then $(u,v,w)(t)$ are radially symmetric for all $t\ge 0$.
\end{itemize}
\end{thm}

A short sketch of the proof of Theorem~\ref{T1:Ex} is contained in Section~\ref{sec.wp}.

\medskip

From now on, given an initial value $(u^0,w^0)$ satisfying~\eqref{a1}, we shall refer to 
\begin{equation*}
    (u,w)=\boldsymbol{\Psi}\big(u^0,w^0\big)
\end{equation*} 
as the (corresponding) solution to~\eqref{E}, with the implicit definition 
\begin{equation*}
    v := \mathcal{K}[w-\langle w\rangle]\,.
\end{equation*}


As pointed out above, system~\eqref{E} has a variational structure, which we describe now, along with an appropriate functional framework. Let $\phi:(0,\infty)\rightarrow (0,\infty)$ be given by
\begin{equation}\label{phi}
	\phi''(z)=m\frac{(1+z)^{m-1}}{z}\,,\quad z>0\,,\qquad \phi(1)=\phi'(1)=0\,.
\end{equation}
Then the functional 
\begin{equation*}
	\mathcal{L}: L_m^+(\Omega)\times  L_p^+(\Omega)\rightarrow \R
\end{equation*}
given by
\begin{equation}\label{FL}
	\mathcal{L}(u,w):=\int_\Omega\big(\phi(u)- u\, \mathcal{K}[ w -\langle  w  \rangle]\big)\,\rd x+\frac{1}{2}\int_\Omega\vert\nabla\mathcal{K}[ w -\langle  w  \rangle]\vert^2\,\rd x
\end{equation}
is well-defined for all $(u,w) \in  L_m^+(\Omega)\times  L_p^+(\Omega)$ whenever $p\ge 2d/(2+d)$. With the notation $v=\mathcal{K}[ w -\langle  w \rangle]$, an equivalent formula for $\mathcal{L}(u,w)$ is
\begin{equation}
\begin{split}
	\mathcal{L}(u,w)&=\int_\Omega\big(\phi(u)- u  v\big)\,\rd x+\frac{1}{2}\int_\Omega\vert\nabla v\vert^2\,\rd x \\
	& =\int_\Omega\big(\phi(u)- u  v\big)\,\rd x+\frac{1}{2}\int_\Omega  w  v\,\rd x\,.
\end{split}\label{i1q}
\end{equation}
We show that $\mathcal{L}$ is a Lyapunov functional for~\eqref{E}. 

\begin{prop}\label{PP2}
Let $p\in (2d,\infty)$ and assume that the initial value $(u^0,w^0)$ satisfies~\eqref{a1}. Then the corresponding solution $(u,w)$ to~\eqref{E} satisfies
\begin{equation*}
\mathcal{L}(u,w)(t)+\int_0^t \mathcal{D}(u,w)(s)\,\rd s\le \mathcal{L}\big(u^0,w^0\big)\,,\quad t\ge 0\,,
\end{equation*}
where
\begin{equation*}
	\mathcal{D}(u,w):=\int_\Omega u\big\vert \nabla\big(\phi'(u)-\mathcal{K}[ w-\langle  w \rangle]\big)\big\vert^2\,\rd x+\|\nabla \mathcal{K}[u-w - \langle u-w\rangle]\|_2^2\ge 0\,.
\end{equation*}
In particular, $\mathcal{L}$ is a strict Lyapunov functional for~\eqref{E}. 
\end{prop}

A proof of Proposition~\ref{PP2} is given in Section~\ref{sec.2}.

\medskip

After these preliminaries we can state our main result concerning the existence of a precisely identified critical mass threshold for~\eqref{E} in a ball $\mathbb{B}_R$ of $\mathbb{R}^d$ with radially symmetric initial data. To this end, the subset of radially symmetric functions of a set $X$ of functions defined on $\mathbb{B}_R$ is denoted by $X_{rad}$, and we recall from Theorem~\ref{T1:Ex} that, for $\bar{M}\ge M>0$, the set 
\begin{equation}
    \mathcal{I}_{M,\bar{M},rad}(\mathbb{B}_R) := \left\{ 
    \begin{array}{l}
    (u,w)\in W_{p,rad}^{1,+}(\mathbb{B}_R) \times L_{\infty,rad}^+(\mathbb{B}_R) \text{ such that }\\[8pt]
    \hspace{1cm} \|u\|_1=M \text{ and } \|w\|_1 \le \bar{M} 
    \end{array}\right\} \label{IS}
\end{equation}
is positively invariant for the semiflow $\boldsymbol{\Psi}$.
 
\begin{thm}\label{THM1}
There is a critical mass $M_c>0$ (defined in~\eqref{Mc} below) such that the following hold:
\begin{itemize}
\item [\textbf{(I)}] For any initial value $(u^0,w^0)\in \mathcal{I}_{M,\bar{M},rad}(\mathbb{B}_R)$ with $M\in (0,M_c)$ and \mbox{$\bar{M}\ge M$}, the global solution $(u,w)$ to~\eqref{E} is bounded.
\item [\textbf{(II)}] Let $M>M_c$. There is $\mu_{M}^s\in\mathbb{R}$ such that, for any $\bar{M}\ge M$, the set
\begin{equation*}
    \mathcal{F}_{M,\bar{M}} := \left\{ (u^0,w^0)\in \mathcal{I}_{M,\bar{M},rad}(\mathbb{B}_R)\ :\ \mathcal{L}(u^0,w^0)<\mu_{M}^s \right\}
\end{equation*}
is non-empty and, for any initial value $(u^0,w^0)\in \mathcal{F}_{M,\bar{M}}$, the global solution $(u,w)$ to~\eqref{E} blows up in infinite time.
\end{itemize}
\end{thm}

\begin{rem}
For $\bar{M}\ge M>M_c$, the above defined set $\mathcal{F}_{M,\bar{M}}$ is non-empty by Proposition~\ref{prop1} below. This property ensures that $M_c$ indeed acts as a threshold.
\end{rem}

The critical mass $M_c$ is characterized as follows: Let 
\begin{equation}
	E_d(x) :=c_d |x|^{2-d}\,, \quad x\in\mathbb{R}^d\setminus\{0\}\,, \qquad c_d:=\frac{1}{(d-2)\sigma_d}\,, \label{Np}
\end{equation}
be the Newton potential, where $\sigma_d:=2\pi^{d/2}/\Gamma(d/2)$ is the surface area of the unit sphere~$\mathbb{S}^{d-1}$. We then note from a variant of the  Hardy–Littlewood–Sobolev inequality \cite[Lemma~3.2]{BCL2009} that the constant
\begin{equation}\label{VHLS}
	C_*:=\frac{1}{c_d}\sup\left\{ \dfrac{\displaystyle\int_{\mathbb{R}^d}h(x) (E_d\ast h)(x)\,\rd x}{\|h\|_{L_1(\mathbb{R}^d)}^{2/d}\,\|h\|_{L_m(\mathbb{R}^d)}^m}\ :\, h\in L_1(\mathbb{R}^d)\cap L_m(\mathbb{R}^d)\,,\, h\not= 0 \right\}
\end{equation}
is finite and positive. We set
\begin{equation}\label{Mc}
	M_c:=\left[\frac{2}{(m-1)c_d C_*}\right]^{d/2}\in (0,\infty)\,,
\end{equation}
which is the critical mass for the quasilinear Patlak–Keller–Segel model with critical nonlinear diffusion \cite{BCL2009}
\begin{equation}
	\partial_t u = \mathrm{div}\big(\nabla u^{m} - u \nabla (E_d*u)\big) \quad \text{ in }\ (0,\infty)\times\mathbb{R}^d \,,\label{qsp}
\end{equation} 
with $m=2(d-1)/d$. An alternative characterization of $M_c$ stems from the existence of minimizers to~\eqref{VHLS}. In fact, 
\begin{subequations}\label{b0}
\begin{equation}\label{b2x}
	M_c = \int_{\mathbb{B}_1}\zeta^{1/(m-1)}(x)\,\rd x\,,
\end{equation}
 where $\zeta$ is the unique radially symmetric classical solution  to 
\begin{equation}\label{b1}
	-\Delta \zeta = \frac{m-1}{m} \zeta^{1/(m-1)}\quad \text{ in }\ \mathbb{B}_1\,,\qquad \zeta=0\quad \text{ on }\ \partial\mathbb{B}_1\,.
\end{equation}
\end{subequations}
Additional properties of $\zeta$ are collected in Appendix~\ref{sec.apB}.

\begin{rem}
   An interesting consequence of Theorem~\ref{THM1} is that the critical mass~$M_c$ obtained therein for~\eqref{E} is the same as that identified for the degenerate quasilinear Patlak-Keller-Segel equation~\eqref{qsp} and is thus not modified by the non-degenerate diffusion in~\eqref{E}. Apparently, the critical mass is only determined by the behavior of the diffusion at infinity. 
\end{rem}

\medskip

\paragraph{\textbf{Outline}} As already mentioned, and in contrast to the approach used in \cite[Corollary~1.5]{FLT2023}, the proof of Theorem~\ref{THM1} relies on the availability of the Lyapunov functional $\mathcal{L}$, as reported in Proposition~\ref{PP2}, and on the strategy developed in \cite{Ho2002,HW2001} for the case $d=2$ and $m=1$. More specifically, given $\bar{M}\ge M>0$, we show that, if $M\in (0,M_c)$, then
\begin{equation}
    \inf_{\mathcal{S}_{M,rad}} \mathcal{L} = \inf_{\mathcal{I}_{M,\bar{M},rad}(\mathbb{B}_R)} \mathcal{L}>-\infty\,, \label{O2}
\end{equation}
while, if $M>M_c$, then
\begin{equation}
    \mu_M^s := \inf_{\mathcal{S}_{M,rad}} \mathcal{L} > \inf_{\mathcal{I}_{M,\bar{M},rad}(\mathbb{B}_R)} \mathcal{L} = -\infty\,, \label{O1}
\end{equation}
where $\mathcal{S}_{M,rad}$ denotes the set of radially symmetric stationary solutions to~\eqref{E} in $\mathcal{I}_{M,\bar{M},rad}(\mathbb{B}_R)$, see~\eqref{Erad} for a precise definition. 

\smallskip

When $M\in (0,M_c)$, the proof of~\eqref{O2} actually provides a control on $\|\nabla v\|_2$ in terms of $\mathcal{L}\big(u^0,w^0\big)$, from which a uniform estimate in $L_m(\mathbb{B}_R)$ on $u$ results, and Theorem~\ref{THM1}~\textbf{(I)} follows with the help of Theorem~\ref{T1:Ex}. Next, once~\eqref{O1} is established for $M>M_c$, we take an arbitrary initial value $(u^0,w^0)\in \mathcal{I}_{M,\bar{M},rad}(\mathbb{B}_R)$ with $\mathcal{L}\big(u^0,w^0\big)<\mu_M^s$. Assuming that the corresponding solution to~\eqref{E} is bounded in $L_\infty\big(\mathbb{B}_R,\mathbb{R}^2\big)$ leads to a contradiction, in view of Proposition~\ref{PP2} and LaSalle's invariance principle which imply stabilization to stationary states. This argument completes the proof of Theorem~\ref{THM1}~\textbf{(II)}.  

Therefore, the cornerstones of the proof of Theorem~\ref{THM1} are the properties~\eqref{O2} and~\eqref{O1}. On the one hand, the radially symmetric setting ensures that the solution $v$ to~\eqref{E2} has constant boundary values $V:=v(x)$, $x\in\partial\mathbb{B}_R$, so that $v-V$ belongs to $H_0^1(\mathbb{B}_R)$ and thus can be extended by zero to $H^1(\mathbb{R}^d)$. This feature enables us to compare $\mathcal{L}$ with a similar functional defined on $\mathbb{R}^d$, for which the variant~\eqref{VHLS} of the Hardy-Littlewood-Sobolev inequality provides a lower bound. On the other hand, fine elliptic estimates in the spirit of \cite{LS1994,WY2003} allow us to discard the existence of an unbounded sequence in $\mathcal{S}_{M,rad}$ provided $M\ne M_c$, from which we deduce~\eqref{O1}.  

\medskip

In Section~\ref{sec.2}, we study the functional $\mathcal{L}$, along with its infinum on $\mathcal{I}_{M,\bar{M},rad}(\mathbb{B}_R)$, as stated in~\eqref{O2} and~\eqref{O1}. Section~\ref{sec.3} is devoted to the proof of Theorem~\ref{THM1}~\textbf{(I)}. Properties of $\mathcal{L}$ on the set of stationary solutions $\mathcal{S}_{M,rad}$ are investigated in Section~\ref{sec.4} for $M\ne M_c$, eventually completing the proof of~\eqref{O1} and of Theorem~\ref{THM1}~\textbf{(II)}. 

\medskip

\paragraph{\textbf{Notation}} We slightly abuse notation by writing $f(x)$, $x\in \bar{\mathbb{B}}_R$, and $f(r)$, $r = |x|\in [0,R]$, simultaneously for radially symmetric functions $f$ defined on the closed ball $\bar{\mathbb{B}}_R$; i.e. we identify radially symmetric functions with their profiles. Also, for further use, we recall that
\begin{equation}
	\int_0^R r^{d-1} z(r)\,\rd r = \frac{\|z\|_1}{\sigma_d} = \frac{R^d}{d} \langle z\rangle\,, \qquad z\in L_{1,rad}^+(\mathbb{B}_R)\,. \label{eqi}
\end{equation}
Finally, from now on, the parameter $p\in (2d,\infty)$ is fixed.

\section{Well-posedness}\label{sec.wp}

We briefly sketch the well-posedness of~\eqref{E} and refer to a forthcoming paper for a detailed proof of a more general result. We additionally provide compactness properties of bounded solutions.

\begin{proof}[Proof of Theorem~\ref{T1:Ex}]
Replacing $v$ in~\eqref{E1} by $v(w):=\mathcal{K}[w-\langle w\rangle]$, the system~\eqref{E} can be recast as a quasilinear Cauchy problem of the form
\begin{equation*}
	\frac{\mathrm{dz}}{\mathrm{d}t}=A(z)z\,,\quad t>0\,,\qquad z(0)=z^0\,,
\end{equation*} 
for $z=(u,w)$ in $W_{p}^1(\Omega)\times L_\infty(\Omega)$, to which the quasilinear theory outlined in~\cite{Am1993} applies to ensure the local existence of a unique (weak) solution. A bootstrapping argument (as in \cite[$\S$~14]{Am1993}) can be used to improve the regularity, while H\"older regularity results for parabolic equations \cite{Am1989, LSU1968} entail that the solution is global, provided that it is bounded on finite time intervals. The latter is in fact implied by the estimates in \cite[Lemma~3.4]{FLT2023}, which also yield that $u$ and $w$ both belong to $L_\infty((0,\infty)\times\Omega)$ if $u\in L_\infty((0,\infty),L_m(\Omega))$. The radial symmetry is a direct consequence of the uniqueness.
\end{proof}

We supplement Theorem~\ref{T1:Ex} with compactness properties of bounded trajectories. Here again, we just sketch the proof and refer to the above mentioned forthcoming paper for details, the main building block of the proof being smoothing properties of parabolic equations as in \cite[15.5~Theorem]{Am1993}.

\begin{lem}\label{Wp}
Consider an initial value $(u^0,w^0)$ satisfying \eqref{a1} and the corresponding solution $(u,w)=\boldsymbol{\Psi}\big(u^0,w^0\big)$ to~\eqref{E} from Theorem~\ref{T1:Ex}. If $u\in L_\infty((0,\infty),L_m(\Omega))$, then $(u,w)([0,\infty))$ is relatively compact in $W_p^1(\Omega)\times L_\infty(\Omega)$.
\end{lem}

\begin{proof}
Owing to~\eqref{regularityA} we may assume without loss of generality that $u^0\in W_{p,N}^2(\Omega)$.  If $u\in L_\infty((0,\infty),L_m(\Omega))$, then $u\in L_\infty((0,\infty)\times \Omega)$ by Theorem~\ref{T1:Ex}. In fact, part of the outcome of the proof of Theorem~\ref{T1:Ex} is the existence of $\eta\in (0,1)$ such that
\begin{equation}\label{uu1}
    u\in BUC^\eta([0,\infty)\times\bar{\Omega})\,,
\end{equation} 
which is actually obtained from H\"older regularity results for parabolic equations, see \cite[Section~5]{Am1989} and~\cite[Chapter~II, Sections~7-8]{LSU1968} for instance. It then readily follows from~\eqref{E3} that $w\in BUC^\eta([0,\infty),L_q(\Omega))$ for any $q\in (1,\infty)$ and thus from~\eqref{E2} that 
\begin{equation}\label{uu2}
    v\in BUC^\eta([0,\infty),W_q^2(\Omega))\,,\quad q\in (1,\infty)\,.
\end{equation} 
Consequently, fixing $2\alpha\in (1,\min\{1+1/p,1+\eta\})$ and setting
\begin{equation*}
    A(t)y:= \mathrm{div}\big(m(1+u(t))^{m-1}\nabla y\big)-\nabla v(t)\cdot\nabla y-\Delta v(t) y\,,
\end{equation*}
we deduce from \cite[Theorem~2.1]{Am1989} (see also \cite{AmE90}) and~\eqref{uu1}-\eqref{uu2} that there exist $\kappa\ge 1$, $\omega>0$ and $\rho>0$ such that
\begin{equation}\label{uu3}
    A(t)\in \mathcal{H}\big(W_{p,N}^{2\alpha}(\Omega),W_{p,N}^{2\alpha-2}(\Omega);\kappa,\omega)\big)\,,\quad t\ge 0\,,
\end{equation}
and
\begin{equation}\label{uu4}
    A\in BUC^\rho\big([0,\infty),\mathcal{L}\big(W_{p,N}^{2\alpha}(\Omega),W_{p,N}^{2\alpha-2}(\Omega)\big)\big)\,.
\end{equation}
Properties~\eqref{uu3} and~\eqref{uu4} allow us to apply the stability estimates from~\cite[Section~II.5]{LQPP} for the evolution operator $U_A$ generated by $A$. We thus infer that there are $c>0$ and $\nu\in \R$ such that
\begin{equation}\label{uu5}
    \|U_A(t,s)\|_{\mathcal{L}(W_{p,N}^{2\alpha-2\delta}(\Omega))}+(t-s)^{\alpha-\delta}\|U_A(t,s)\|_{\mathcal{L}(L_p(\Omega),W_{p,N}^{2\alpha-2\delta}(\Omega))}\le c e^{\nu(t-s)}
\end{equation}
for $t\ge s\ge 0$, where $\delta>0$ is chosen such that $1<2\alpha-2\delta<2\alpha$. Choosing $\nu_0>\nu$ and noticing that $u$ solves
\begin{equation*}
    \frac{\rd u}{\rd t}(t) =\big(A(t)-\nu_0\big)u(t)+\nu_0 u(t)\,,\quad t>0\,,\qquad u(0)=u^0\in W_{p,N}^2(\Omega)\,,
\end{equation*}
it  readily follows from~\eqref{uu5} that $u([0,\infty))$ is bounded in  $W_{p,N}^{2\alpha-2\delta}(\Omega)$, and hence relatively compact in $W_{p}^{1}(\Omega)$. Finally, the H\"older continuity~\eqref{uu1} of $u$ implies that the mapping
\begin{equation*}
    W:\left[t\mapsto \int_0^t e^{-(t-s)}u(s)\,\rd s\right]
\end{equation*}
is bounded in $C^\eta(\bar{\Omega})$, hence $W([0,\infty))$ is relatively compact in $C(\bar{\Omega})$. Since 
\begin{equation*}
    w(t)=e^{-t}w^0+W(t)\,,\quad t\ge 0\,,
\end{equation*} 
we conclude that $w([0,\infty))$ is relatively compact in $L_\infty(\Omega)$ as claimed.
\end{proof}

\section{A Lyapunov functional}\label{sec.2}

In this section, we focus on the functional $\mathcal{L}$ defined in~\eqref{FL}. We first prove that it is indeed a Lyapunov functional for~\eqref{E}, as stated in Proposition~\ref{PP2}, and then investigate its boundedness or unboundedness from below. We begin by establishing the continuity of $\mathcal{L}$ on $L_m^+(\Omega)\times  L_p^+(\Omega)$.

\begin{lem}\label{lem.contL} 
Let $p\in (2d,\infty)$. Then $\mathcal{L}$ belongs to $C\big(L_m^+(\Omega)\times  L_p^+(\Omega)\big)$. 
\end{lem}

\begin{proof}
The bound~\eqref{eqA1} and the Lebesgue dominated convergence theorem entail that
\begin{equation}
   [u\mapsto \phi(u)]\in C\big(L_m^+(\Omega),L_1(\Omega)\big)\,. \label{A}
\end{equation}
Moreover, we infer from~\eqref{opK} and the continuous embedding of $W_p^2(\Omega)$ into $W_2^1(\Omega)$ (since $p\ge 2d/(2+d)$) that
\begin{equation}
   \big[w\mapsto \mathcal{K}[w-\langle w\rangle]\big]\in \mathcal{L}\big(L_p(\Omega),W_p^2(\Omega)\big)\subset \mathcal{L}\big(L_p(\Omega),W_2^1(\Omega)\big)\,, \label{B}
\end{equation}
which, together with the continuous embedding of $W_p^2(\Omega)$ into  $L_{m/(m-1)}(\Omega)$, implies that
\begin{equation}
   \big[(u,w)\mapsto u\mathcal{K}[w-\langle w\rangle]\big]\in C\big(L_m(\Omega)\times L_p(\Omega),L_1(\Omega)\big)\,. \label{C}
\end{equation}
Recalling~\eqref{FL}, it follows from \eqref{A}-\eqref{C} that $\mathcal{L}$ belongs to $C\big(L_m^+(\Omega)\times  L_p^+(\Omega)\big)$ as claimed.
\end{proof}

\begin{proof}[Proof of Proposition~\ref{PP2}]
To cope with the singularity of $\phi'(r)$ as $r\to 0^+$ we consider $u_\ve^0:=u^0+\ve$ with
 $\ve>0$. Let $(u_\ve,w_\ve)$ be the corresponding global solution to~\eqref{E} on $[0,\infty)$ with initial value $\big(u_\ve^0,w^0\big)$ and  $v_\ve = \mathcal{K}[w_\varepsilon - \langle w_\varepsilon\rangle]$. Fix $T>0$ and set  $\gamma_\ve(T) :=\|\Delta v_\ve\|_{L_\infty((0,T)\times\Omega)}$, which is well-defined due to~\eqref{E2} and~\eqref{regularityB} Then, $u_\ve$ satisfies
\begin{align*}
	\partial_tu_\ve & = \big(\mathcal{A}_\ve(t) - \gamma_\ve(T) \big)u_\ve+\big(\gamma_\ve(T) - \Delta v_\ve(t)\big) u_\ve \ge \big(\mathcal{A}_\ve(t)-\gamma_\ve(T)\big)u_\ve\,, \\
	u_\ve(0) & = u_\ve^0\ge \ve\,,
\end{align*}
with
\begin{equation*}
	\mathcal{A}_\ve(t)z:=\mathrm{div}\big(m(1+u_\ve(t))^{m-1}\nabla z\big)-\nabla v_\ve(t)\cdot\nabla z
\end{equation*}
and is subject to Neumann boundary conditions. By the comparison principle, 
\begin{equation*}
	u_\ve(t)\ge e^{-\gamma_\ve(T) t}\ve>0\,,\quad (t,x)\in [0,T]\times \bar\Omega\,.
\end{equation*}
Since $p>2>m$, it follows from~\eqref{regularityA} and~\eqref{phi} that 
\begin{equation*}
	\phi(u_\ve)\in C^1 \big((0,T],L_1(\Omega)\big)\,,\qquad \partial_t\phi(u_\ve)=\phi'(u_\ve)\partial_tu_\ve\,.
\end{equation*}
Moreover, $W_p^{1}(\Omega)$ is continuously embedded in $C(\bar{\Omega})$ and we infer from \eqref{regularityA}-\eqref{regularityB} that 
\begin{equation}\label{i1}
	\mathcal{L}(u_\ve,w_\ve)\in C^1((0,\infty))\cap C([0,\infty))\,,\qquad \mathcal{D}(u_\ve,w_\ve)\in  C([0,\infty))\,.
\end{equation}
Since $v_\ve=\mathcal{K}[w_\ve-\langle w_\ve\rangle]\in C^1([0,\infty),W_{p,N}^2(\Omega))$ with
\begin{equation*}
	\int_\Omega\partial_t v_\ve\,\rd x=0\,,
\end{equation*}
we derive from~\eqref{E}, \eqref{phi} and~\eqref{i1q} that, for $t\in (0,T]$,
\begin{align*}
	\frac{\rd}{\rd t}\mathcal{L}(u_\ve,w_\ve)&= \int_\Omega\big(\phi'( u_\ve)- v_\ve\big)\partial_t u_\ve\,\rd x-\int_\Omega u_\ve\partial_t v_\ve\,\rd x+\int_\Omega \nabla v_\ve\cdot\nabla\partial_t v_\ve\,\rd x\\
	&= -\int_\Omega u_\ve \big\vert\nabla\big(\phi'( u_\ve)- v_\ve\big)\big\vert^2\,\rd x-\int_\Omega \big(u_\ve-\langle u_\ve\rangle\big)\partial_t v_\ve\,\rd x-\int_\Omega \Delta v_\ve\partial_t v_\ve\,\rd x\\
	&= -\int_\Omega u_\ve \big\vert\nabla\big(\phi'( u_\ve)- v_\ve\big)\big\vert^2\,\rd x-\int_\Omega \big(u_\ve-w_\ve - \langle u_\ve -w_\ve \rangle\big) \partial_t v_\ve\,\rd x\,.
\end{align*}
Using the identities (see ~\eqref{E3})
\begin{align*}
	-\Delta\mathcal{K}\big[u_\ve-w_\ve - \langle u_\ve -w_\ve \rangle\big] & = u_\ve-w_\ve - \langle u_\ve -w_\ve \rangle\,, \\
	\partial_t v_\ve & = \mathcal{K}[ u_\ve-w_\ve-\langle u_\ve-w_\ve\rangle]\,, 
\end{align*}
along with the divergence theorem for the last term, we find that, for $t\in (0,T]$, 
\begin{align*}
	\frac{\rd}{\rd t}\mathcal{L}(u_\ve,w_\ve)& = -\int_\Omega u_\ve \big\vert\nabla\big(\phi'( u_\ve)- v_\ve\big)\big\vert^2\,\rd x-\int_\Omega \big\vert\nabla \mathcal{K}\big[u_\ve - w_\ve- \langle u_\ve - w_\ve\rangle\big]\big\vert^2\,\rd x\\
&=-\mathcal{D}(u_\ve,w_\ve)\,.
\end{align*}
Therefore,
\begin{equation*}
	\mathcal{L}(u_\ve,w_\ve)(t) + \int_s^t \mathcal{D}(u_\ve,w_\ve)(\tau)\,\rd\tau = \mathcal{L}(u_\ve,w_\ve)(s)
\end{equation*}
for $0<s<t\le T$. Since $ \mathcal{D}(u_\ve,w_\ve)$ and $\mathcal{L}(u_\ve,w_\ve)$ both belong to  $C([0,\infty))$ by~\eqref{i1}, we may let $s\to 0$ to obtain
\begin{equation*}
	\mathcal{L}(u_\ve,w_\ve)(t) + \int_0^t \mathcal{D}(u_\ve,w_\ve)(\tau)\,\rd\tau = \mathcal{L}\big(u_\ve^0,w^0\big)
\end{equation*}
for $t\in (0,T]$. Fatou's lemma and the continuous dependence of the solution to~\eqref{E} on the initial value stated in Theorem~\ref{T1:Ex}  ensure that
\begin{align*}
	\int_0^t \mathcal{D}(u,w)(\tau)\,\rd\tau & \le  \liminf_{\ve\to 0} \int_0^t \mathcal{D}(u_\ve,w_\ve)(\tau)\,\rd\tau\\
	&=\liminf_{\ve\to 0} \left( \mathcal{L}\big(u_\ve^0,w^0\big) - \mathcal{L}(u_\ve,w_\ve)(t) \right) = \mathcal{L}\big(u^0,w^0\big) - \mathcal{L}(u,w)(t)
\end{align*}
for $t\ge 0$ (since $T>0$ was arbitrary). This proves Proposition~\ref{PP2}. 
\end{proof}

We next turn to lower bounds on $\mathcal{L}$ on the set $\mathcal{I}_{M,\bar{M},rad}(\mathbb{B}_R)$ according to the range of $(M,\bar{M})$. To this end, we take advantage of radial symmetry, which guarantees that the solution $v$ to~\eqref{E2} supplemented with the homogeneous Neumann boundary conditions~\eqref{E4} is also a solution to~\eqref{E2} supplemented with nonhomogeneous \textit{constant} Dirichlet boundary conditions. In this direction, we first report the following estimate.

\begin{lem}\label{LL1}
Let $w\in L_{1,rad}^+(\mathbb{B}_R)$. Then $v=\mathcal{K}[w-\langle w\rangle]$ is radially symmetric and  $|v(R)| \le c_0 \langle w\rangle$ with $c_0=c_0(d,R)>0$.
\end{lem}

\begin{proof}
The radial symmetry of $v$ is obvious due to the well-posedness of~\eqref{opKb} and the rotational invariance of the Laplace operator and the mean value. In radial variables, the equation~\eqref{opKb} reads
\begin{equation*}
	\partial_r\left(r^{d-1}\partial_rv(r)\right)=r^{d-1}\left( \langle w\rangle - w(r) \right)\,,\quad 0<r<R\,,
\end{equation*}
from which we derive, after integrating twice,
\begin{align*}
	v(R)-v(r)&=\frac{\langle w\rangle}{2d}\big(R^2-r^2\big)-\int_r^Rs^{1-d}\int_0^sz^{d-1} w(z)\,\rd z\rd s\\
	&=\frac{\langle w\rangle}{2d}\big(R^2-r^2\big)+\frac{R^{2-d}}{d-2}\int_0^R z^{d-1} w(z)\,\rd z\\
	&\quad - \frac{1}{d-2} \int_r^R z w(z)\rd z - \frac{r^{2-d}}{d-2} \int_0^r z^{d-1} w(z)\,\rd z\,.
\end{align*}
Therefore, multiplying by $r^{d-1}$, integrating with respect to $r\in (0,R)$, and recalling that
\begin{equation*}
	\int_0^R r^{d-1} v(r)\,\rd r=0
\end{equation*} 
yields
\begin{align*}
	\frac{R^d}{d}v(R) & = \frac{\langle w\rangle}{2d} R^{d+2} \left(\frac{1}{d}-\frac{1}{d+2}\right) + \frac{R^2}{d(d-2)} \int_0^Rz^{d-1} w(s)\,\rd s\\
	& \quad -\frac{1}{d-2} \int_0^R  s w(s) \int_0^s r^{d-1}\,\rd r\rd s - \frac{1}{d-2} \int_0^R s^{d-1} w(s) \int_s^R r\,\rd r\rd s\\
	& = \frac{\langle w\rangle}{2d} R^{d+2} \left(\frac{1}{d}-\frac{1}{d+2}\right) + \frac{R^2}{d-2} \left(\frac{1}{d}-\frac{1}{2}\right) \int_0^R s^{d-1} w(s)\,\rd s\\
	& \quad - \frac{1}{d-2} \left(\frac{1}{d}-\frac{1}{2}\right) \int_0^R s^{d+1} w(s)\,\rd s\,,
\end{align*}
from which we obtain that
\begin{equation*}
	v(R) = \frac{\langle w\rangle R^2}{d(d+2)} -\frac{1}{2R^d} \int_0^R s^{d-1} \left(R^2-s^2\right) w(s)\,\rd s\,.
\end{equation*}
Consequently, thanks to the non-negativity of $w$ and~\eqref{eqi}, we conclude that
\begin{equation*}
	-\frac{\langle w\rangle R^2 }{2(d+2)}\le v(R)\le\frac{\langle w\rangle R^2}{d(d+2)} \,.
\end{equation*}
hence the assertion.
\end{proof}
 
\begin{lem}\label{LLL2}
Let $(u,w)\in  L_{m,rad}^+(\mathbb{B}_R)\times  L_{p,rad}^+(\mathbb{B}_R)$ be such that $M=\|u\|_1\le M_c$ and set $v= \mathcal{K}[ w -\langle  w \rangle]$. Then
\begin{equation*}
	\mathcal{L}(u,w)\ge \frac{1}{2}\left(1-\left(\frac{M}{M_c}\right)^{2/d}\right) \|\nabla v\|_2^2 -M c_0\langle w\rangle -\frac{m M}{m-1}\,,
\end{equation*}
where $c_0>0$ is the constant from Lemma~\ref{LL1}.
\end{lem}

\begin{proof}
We set $V:= v(R)$ and define the trivial extension $\bar{u}$ of $u$ to $\mathbb{R}^d$ by $\bar{u}(x)=u(x)$ for $x\in\mathbb{B}_R$ and $\bar{u}(x) = 0$ for $x\in\mathbb{R}^d\setminus \mathbb{B}_R$. Owing to the properties~\eqref{Np} of the Newton potential, 
\begin{align*}
	\int_{\mathbb{B}_R} u v\,\rd x & = \int_{\mathbb{B}_R} \bar{u} (v-V)\,\rd x+VM = -\int_{\mathbb{B}_R} (v-V)\,\Delta (E_d*\bar{u})\, \rd x + VM\,.
\end{align*}
Since $v-V=0$ on $\partial\mathbb{B}_R$ we obtain, using the divergence theorem,~\eqref{VHLS}, and~\eqref{Mc}, along with H\"older's and Young's inequalities, that
\begin{align}
	\int_{\mathbb{B}_R} u v\,\rd x & = \int_{\mathbb{B}_R} \nabla( E_d*\bar{u}) \cdot \nabla v\,\rd x+VM \le \|\nabla( E_d*\bar{u}) \|_{L_2(\mathbb{R}^d)}\,\|\nabla v\|_{2} +VM \nonumber\\
	& \le \left( \int_{\mathbb{R}^d} \bar{u} \big[E_d*\bar{u}\big] \right)^{1/2} \,\|\nabla v\|_{2} +VM \nonumber\\
	&\le \sqrt{C_* c_d}\, \|\bar{u}\|_{L_1(\R^d)}^{1/d}\, \|\bar{u}\|_{L_m(\R^d)}^{m/2}\,\|\nabla v\|_{2} +VM\nonumber\\
	&=\sqrt{\frac{2}{m-1}}\, \left(\frac{M}{M_c}\right)^{1/d}\, \|u\|_{m}^{m/2}\,\|\nabla v\|_{2} +VM\nonumber\\
	&\le \frac{1}{2}\, \left(\frac{M}{M_c}\right)^{2/d}\,\|\nabla v\|_{2}^2+ \frac{1}{m-1}\|u\|_{m}^{m}\, +VM\,.\label{r101}
\end{align}
Since
\begin{equation*}
	\phi(s)\ge \frac{s^{m}}{m-1}+1-\frac{m}{m-1}s\,,\quad s\ge 0\,,
\end{equation*}
by~\eqref{eq.A2}, we infer from~\eqref{r101} and the above inequality that
\begin{align*}
	\mathcal{L}(u,w)&=\int_{\mathbb{B}_R} \phi(u)\,\rd x-\int_{\mathbb{B}_R} u v\,\rd x+\frac{1}{2}\|\nabla v\|_{2}^2\\
&\ge \frac{1}{2} \left(1-\left(\frac{M}{M_c}\right)^{2/d}\right)\,\|\nabla v\|_{2}^2- VM  -\frac{m}{m-1}M \,.
\end{align*}
Since $|V| \le c_0\langle w\rangle $ according to Lemma~\ref{LL1}, the assertion follows.
\end{proof}

Based on the properties of the functional $\mathcal{L}$ established so far, we can now establish the main result of this section.

\begin{prop}\label{prop1}
Introducing
\begin{equation}
    \mu_{M,\bar{M}} := \inf_{(u,w)\in \mathcal{I}_{M,\bar{M},rad}(\mathbb{B}_R)} \mathcal{L}(u,w) \in [-\infty,\infty)\,, \quad \bar{M}\ge M>0\,, \label{mu}
\end{equation}
with $\mathcal{I}_{M,\bar{M},rad}(\mathbb{B}_R)$ defined in~\eqref{IS}, there holds:
\begin{equation}
\begin{split}
    \mu_{M,\bar{M}} & = -\infty \;\;\text{ for }\;\; \bar{M}\ge M>M_c\,,\\ 
    \mu_{M,\bar{M}} & > -\infty \;\;\text{ for }\;\; M\in (0,M_c] \;\text{ and }\;\bar{M}\ge M\,.
\end{split}\label{x1}
\end{equation}
More precisely, in the second case, there is $c_0=c_0(d,R)>0$ such that, for all $(u,w)\in \mathcal{I}_{M,\bar{M},rad}(\mathbb{B}_R)$,
\begin{equation}
\begin{split}
    \mathcal{L}(u,w) \ge \frac{1}{2} \left( 1 - \left( \frac{M}{M_c} \right)^{2/d} \right) \|\nabla\mathcal{K}[w-\langle w\rangle]\|_2^2 - c_0 \frac{M\bar{M}}{|\mathbb{B}_R|} - \frac{mM}{m-1}\,. \label{lbL}
\end{split}
\end{equation}
\end{prop}

\begin{proof}
For $M\in (0,M_c]$ and $\bar{M}\ge M$, Proposition~\ref{prop1} readily follows from Lemma~\ref{LLL2} and the embedding of $W_p^1(\mathbb{B}_R)$ into $L_m(\mathbb{B}_R)$.

We next turn to the case $\bar{M}\ge M>M_c$ and exploit the fact that the energy associated with the initial value problem~\eqref{qsp} is unbounded from below \cite[Proposition~3.4]{BCL2009}. However, being in a bounded domain with a non-degenerate diffusion, it is not possible to apply directly the analysis performed in \cite{BCL2009} but it rather has to be adapted to our setting, in particular to comply with the boundary conditions. We set $\kappa:= M/M_c>1$ and define, for $\lambda>\max\{1,1/R\}$ and $\zeta$ given in~\eqref{b1},
\begin{subequations}\label{fl}
\begin{align}
    u_{\lambda}(x) & := \kappa \lambda^d \zeta^{1/(m-1)}(\lambda x)\,, \qquad x\in \mathbb{B}_R\,, \label{ul}\\
    V_{\lambda}(x) & := \kappa \lambda^{d-2} \big(E_d*\zeta^{1/(m-1)}\big)(\lambda x) + \frac{M}{2\sigma_d R^d} |x|^2\,, \qquad x\in \mathbb{B}_R\,, \label{Vl} \\
    v_{\lambda}(x) & := V_{\lambda}(x) - \langle V_{\lambda} \rangle\,, \qquad x\in \mathbb{B}_R\,, \label{vl}\\
    w_{\lambda}(x) & := u_{\lambda}(x)\,, \qquad x\in\mathbb{B}_R\,. \label{wl}
\end{align}
\end{subequations}
We first collect several properties of the functions defined in~\eqref{fl} and, in particular, check that $(u_{\lambda},w_{\lambda})$ belongs to $\mathcal{I}_{M,\bar{M},rad}(\mathbb{B}_R)$. Indeed, $u_{\lambda}=w_{\lambda}$ is clearly a radially symmetric function in $W_p^{1,+}(\mathbb{B}_R)$ with compact support in $\bar{\mathbb{B}}_{1/\lambda}\subset \mathbb{B}_R$ owing to $\lambda R>1$ and $\mathrm{supp}\, \zeta = \bar{\mathbb{B}}_1$. Moreover, since $\mathbb{B}_1\subset \mathbb{B}_{\lambda R}$, we deduce from~\eqref{b2x} that
\begin{equation}
	\|u_\lambda\|_1 = \|w_\lambda\|_1 = \kappa \int_{\mathbb{B}_{\lambda R}} \zeta^{1/(m-1)}(y)\,\rd y = \frac{M}{M_c} \int_{\mathbb{B}_{1}} \zeta^{1/(m-1)}(y)\,\rd y = M\,. \label{ub01}
\end{equation}
Consequently,
\begin{equation}
    (u_{\lambda},w_{\lambda}) \in \mathcal{I}_{M,\bar{M},rad}(\mathbb{B}_R)\,. \label{ub00}
\end{equation}
We next infer from~\eqref{eq.B3} and the constraint $\lambda R>1$ that, for $x\in \mathbb{B}_R\setminus \mathbb{B}_{1/\lambda}$, 
\begin{align*}
    v_{\lambda}(x) & = \frac{\kappa M_c \lambda^{d-2}}{(d-2)\sigma_d} \frac{1}{|\lambda x|^{d-2}} + \frac{M}{2\sigma_d R^d} |x|^2 - \langle V_{\lambda}\rangle \\
    & = \frac{M}{(d-2)\sigma_d} \frac{1}{|x|^{d-2}} + \frac{M}{2\sigma_d R^d} |x|^2 - \langle V_{\lambda}\rangle\,,
\end{align*}
and thus
\begin{equation*}
    \nabla v_{\lambda}(x) = - \frac{M}{\sigma_d} \frac{x}{|x|^d} + \frac{M}{\sigma_d R^d} x\,, \qquad x\in \mathbb{B}_R\setminus \mathbb{B}_{1/\lambda}\,.
\end{equation*}
Therefore,
\begin{equation}
    \nabla v_{\lambda}(x)\cdot \mathbf{n} = 0\,, \qquad x\in\partial\mathbb{B}_R\,. \label{ub02}
\end{equation}
Moreover, for $x\in\mathbb{B}_R$, 
\begin{align*}
    - \Delta v_{\lambda}(x) & = - \kappa \lambda^d \Delta\big(E_d*\zeta^{1/(m-1)}\big)(\lambda x) - \frac{dM}{\sigma_d R^d} = \kappa \lambda^d \zeta^{1/(m-1)}(\lambda x) - \frac{M}{|\mathbb{B}_R|}\,,
\end{align*}
hence, by~\eqref{wl} and~\eqref{ub01},
\begin{equation}
    - \Delta v_{\lambda} = w_\lambda - \langle w_\lambda\rangle \;\text{ in }\; \mathbb{B}_R\,. \label{ub03}
\end{equation}
An immediate consequence of~\eqref{ub02} and~\eqref{ub03} is that
\begin{equation}
    v_{\lambda} = V_{\lambda} - \langle V_{\lambda}\rangle = \mathcal{K}[w_\lambda - \langle w_\lambda\rangle]\,. \label{ub04} 
\end{equation}
Thanks to~\eqref{ub01} and~\eqref{ub04},
\begin{align}
    \mathcal{L}(u_{\lambda},w_{\lambda}) & = \int_{\mathbb{B}_R} \phi(u_{\lambda})(x)\,\rd x - \int_{\mathbb{B}_R} (u_{\lambda} v_{\lambda})(x)\,\rd x + \frac{1}{2} \int_{\mathbb{B}_R} |\nabla v_{\lambda}(x)|^2\;\rd x \nonumber\\
    & = \int_{\mathbb{B}_R} \phi(u_{\lambda})(x)\,\rd x + M \langle V_{\lambda}\rangle - \int_{\mathbb{B}_R} (u_{\lambda} V_{\lambda})(x)\,\rd x \label{ub05} \\
    & \qquad + \frac{1}{2} \int_{\mathbb{B}_R} |\nabla v_{\lambda}(x)|^2\,\rd x\,, \nonumber 
\end{align}
and we now compute and estimate separately each term on the right-hand side of~\eqref{ub05}. To this end, let $\varepsilon\in (0,1)$ to be specified later. We first infer from~\eqref{eqA1} that
\begin{align}
    \int_{\mathbb{B}_R} \phi(u_{\lambda})(x)\,\rd x & = \frac{1}{\lambda^d} \int_{\mathbb{B}_{\lambda R}} \phi\big(\kappa \lambda^d \zeta^{1/(m-1)}\big)(y)\,\rd y \nonumber \\
    & \le \frac{(1+\varepsilon)\kappa^m \lambda^{d-2}}{m-1} \int_{\mathbb{B}_{\lambda R}} \zeta^{m/(m-1)}(y)\,\rd y + c_\phi(\varepsilon) \frac{|\mathbb{B}_{\lambda R}|}{\lambda^d} \nonumber \\
    & \le \frac{(1+\varepsilon)\kappa^m \lambda^{d-2}}{m-1} I_\zeta + |\mathbb{B}_R| c_\phi(\varepsilon)\,, \label{ub06}
\end{align}
where $I_\zeta$ is defined in~\eqref{eq.B4}.

Next, by~\eqref{ul},~\eqref{Vl}, and~\eqref{eq.B6}, 
\begin{align}
   \int_{\mathbb{B}_R} (u_{\lambda} V_{\lambda})(x)\,\rd x & = \kappa^2 \lambda^{2d-2} \int_{\mathbb{B}_R} \big[\zeta^{1/(m-1)} \big(E_d*\zeta^{1/(m-1)}\big)\big](\lambda x)\,\rd x \nonumber\\
   & \qquad + \frac{M\kappa \lambda^d}{2\sigma_d R^d} \int_{\mathbb{B}_R} |x|^2 \zeta^{1/(m-1)}(\lambda x)\,\rd x \nonumber \\
   & \ge \kappa^2 \lambda^{d-2} \int_{\mathbb{B}_{\lambda R}} \big[\zeta^{1/(m-1)} \big(E_d*\zeta^{1/(m-1)}\big)\big](y)\,\rd y \nonumber\\
   & = \kappa^2 \lambda^{d-2} \int_{\mathbb{B}_1} \big[\zeta^{1/(m-1)} \big(E_d*\zeta^{1/(m-1)}\big)\big](y)\,\rd y \nonumber \\
   & = \frac{2\kappa^2 \lambda^{d-2}}{m-1} I_\zeta\,. \label{ub07}
\end{align}
Also, by~\eqref{Vl}, \eqref{eq.B2} and~\eqref{eq.B3}, 
\begin{align*}
    M \langle V_{\lambda} \rangle & = \frac{M\kappa \lambda^{d-2}}{|\mathbb{B}_R|} \int_{\mathbb{B}_R} \big(E_d*\zeta^{1/(m-1)}\big)(\lambda x)\,\rd x + \frac{M^2}{2\sigma_d R^d |\mathbb{B}_R|} \int_{\mathbb{B}_R} |x|^2\,\rd x \\
    & = \frac{M\kappa}{\lambda^2|\mathbb{B}_R|} \int_{\mathbb{B}_{\lambda R}} \big(E_d*\zeta^{1/(m-1)}\big)(y)\,\rd y + \frac{M^2 R^2}{2(d+2) |\mathbb{B}_R|} \\
    & = \frac{M\kappa}{\lambda^2|\mathbb{B}_R|} \int_{\mathbb{B}_1} \left( \frac{m}{m-1} \zeta(y) + \frac{2-m}{m-1} \frac{I_\zeta}{M_c} \right)\,\rd y \\
    & \qquad + \frac{M\kappa}{\lambda^2|\mathbb{B}_R|} \frac{M_c}{(d-2)\sigma_d} \int_{\mathbb{B}_{\lambda R}\setminus\mathbb{B}_1} \frac{\rd y}{|y|^{d-2}}  + \frac{M^2 R^2}{2(d+2) |\mathbb{B}_R|} \\
    & \le \frac{M\kappa R^2}{|\mathbb{B}_R|} \int_{\mathbb{B}_1} \left( \frac{m}{m-1} \zeta(y) + \frac{2-m}{m-1} \frac{I_\zeta}{M_c} \right)\,\rd y\\
    & \qquad + \frac{M^2 R^2}{2(d-2)|\mathbb{B}_R|} + \frac{M^2 R^2}{2(d+2) |\mathbb{B}_R|} \,,
\end{align*}
hence
\begin{equation}
    M \langle V_{\lambda} \rangle \le c_1\,,   \label{ub08}
\end{equation}
with
\begin{equation*}
    c_1 := \frac{M\kappa R^2}{|\mathbb{B}_R|} \int_{\mathbb{B}_1} \left( \frac{m}{m-1} \zeta(y) + \frac{2-m}{m-1} \frac{I_\zeta}{M_c} \right)\,\rd y + \frac{d M^2 R^2}{(d-2)(d+2)|\mathbb{B}_R|} \,.
\end{equation*}
Finally, by Young's inequality,
\begin{align}
    & \frac{1}{2} \int_{\mathbb{B}_R} |\nabla v_{\lambda}(x)|^2\,\rd x = \frac{1}{2} \int_{\mathbb{B}_R} |\nabla V_{\lambda}(x)|^2\,\rd x \nonumber \\
    & \quad \le \frac{(1+\varepsilon)\kappa^2 \lambda^{2d-2}}{2} \int_{\mathbb{B}_R} \left|\nabla\big(E_d*\zeta^{1/(m-1)}\big)(\lambda x)\right|^2\,\rd x + \frac{1+\varepsilon}{2\varepsilon} \frac{M^2}{\sigma_d^2 R^{2d}} \int_{\mathbb{B}_R} |x|^2\,\rd x \nonumber \\
    & \quad = \frac{(1+\varepsilon)\kappa^2 \lambda^{d-2}}{2} \int_{\mathbb{B}_{\lambda R}} \left|\nabla\big(E_d*\zeta^{1/(m-1)}\big)(y)\right|^2\,\rd y + \frac{1+\varepsilon}{2(d+2)\varepsilon} \frac{M^2}{\sigma_d R^{d-2}} \nonumber \\
    & \quad = \frac{(1+\varepsilon)\kappa^2 \lambda^{d-2}}{2} \int_{\mathbb{B}_1} \left|\nabla\big(E_d*\zeta^{1/(m-1)}\big)(y)\right|^2\,\rd y \label{ub09} \\
    & \quad\quad + \frac{(1+\varepsilon)\kappa^2 \lambda^{d-2}}{2} \int_{\mathbb{B}_{\lambda R}\setminus\mathbb{B}_1} \left|\nabla\big(E_d*\zeta^{1/(m-1)}\big)(y)\right|^2\,\rd y + \frac{c_2}{\varepsilon} \nonumber \,,
\end{align}
with $c_2 := M^2/(\sigma_d R^{d-2})$. Now, it follows from~\eqref{eq.B2} and~\eqref{eq.B5} that
\begin{equation*}
    \int_{\mathbb{B}_1} \left|\nabla\big(E_d*\zeta^{1/(m-1)}\big)(y)\right|^2\,\rd y = \left( \frac{m}{m-1} \right)^2 \int_{\mathbb{B}_1} |\nabla\zeta(y)|^2\,\rd y = \frac{m}{m-1} I_\zeta
\end{equation*}
and from~\eqref{eq.B3} and~\eqref{eq.B5} that
\begin{align*}
    \int_{\mathbb{B}_{\lambda R}\setminus\mathbb{B}_1} \left|\nabla\big(E_d*\zeta^{1/(m-1)}\big)(y)\right|^2\,\rd y & = \frac{M_c^2}{\sigma_d^2} \int_{\mathbb{B}_{\lambda R}\setminus\mathbb{B}_1} \frac{\rd y}{|y|^{2d-2}} \\
    & \le \frac{M_c^2}{\sigma_d} \int_1^\infty r^{1-d}\,\rd r = \frac{M_c^2}{(d-2)\sigma_d} \\
    & = \frac{2-m}{m-1} I_\zeta\,.
\end{align*}
Inserting the above estimates in~\eqref{ub09} leads us to 
\begin{equation*}
    \frac{1}{2} \int_{\mathbb{B}_R} |\nabla v_{\lambda}(x)|^2\,\rd x \le \frac{(1+\varepsilon)\kappa^2 \lambda^{d-2}}{m-1} I_\zeta + \frac{c_2}{\varepsilon}\,.
\end{equation*}
Combining the above estimate with~\eqref{ub05}, \eqref{ub06}, \eqref{ub07} and~\eqref{ub08} gives
\begin{align}
   \mathcal{L}(u_{\lambda},w_{\lambda}) & \le  \frac{(1+\varepsilon)\kappa^m \lambda^{d-2}}{m-1} I_\zeta + |\mathbb{B}_R| c_\phi(\varepsilon)  + c_1 - \frac{2\kappa^2 \lambda^{d-2}}{m-1} I_\zeta + \frac{(1+\varepsilon)\kappa^2 \lambda^{d-2}}{m-1} I_\zeta + \frac{c_2}{\varepsilon} \nonumber \\
   & = \frac{\kappa^2 \lambda^{d-2}}{m-1} I_\zeta \left[ \frac{1+\varepsilon}{\kappa^{2-m}} - (1 - \varepsilon) \right] + |\mathbb{B}_R| c_\phi(\varepsilon)  + c_1 + \frac{c_2}{\varepsilon}\,. \label{ub10}
\end{align}
Now, since $\kappa>1$ due to $M>M_c$, we choose
\begin{equation*}
    \varepsilon = \varepsilon_\kappa := \frac{1 - \kappa^{m-2}}{3+\kappa^{m-2}} \in (0,1)\,,
\end{equation*}
so that
\begin{equation*}
    \frac{1+\varepsilon_\kappa}{\kappa^{2-m}} = \frac{4\kappa^{m-2}}{3+\kappa^{m-2}}\,, \quad 1 - \varepsilon_\kappa = \frac{2+2\kappa^{m-2}}{3+\kappa^{m-2}}  
\end{equation*}
and thus
\begin{equation*}
   \frac{1+\varepsilon_\kappa}{\kappa^{2-m}} - (1 - \varepsilon_\kappa) = - \frac{2-2\kappa^{m-2}}{3+\kappa^{m-2}}\,. 
\end{equation*}
Choosing $\varepsilon=\varepsilon_\kappa$ in~\eqref{ub10}, we end up with
\begin{equation*}
    \mathcal{L}(u_{\lambda},w_{\lambda}) \le - \frac{2\kappa^2 \lambda^{d-2}}{m-1} \frac{1-\kappa^{m-2}}{3+\kappa^{m-2}} I_\zeta + |\mathbb{B}_R| c_\phi(\varepsilon_\kappa) + c_1 + \frac{c_2}{\varepsilon_\kappa}\,.
\end{equation*}
Since $d\ge 3$, it readily follows from~\eqref{ub00} and the above inequality that
\begin{equation*}
    \mu_{M,\bar{M}} \le \liminf_{\lambda\to\infty} \mathcal{L}(u_{\lambda},w_{\lambda}) = -\infty\,,
\end{equation*}
and the proof is complete.
\end{proof}

\begin{rem}
In the semilinear case $m=1$ in space dimension $d=2$, similar arguments are used in \cite{La2019} to establish that the corresponding functional is not bounded from below, following the approach developed in \cite{FJ2022, Ho2002, HW2001}. 
\end{rem}

\section{Proof of Theorem~\ref{THM1}~\textbf{(I)}}\label{sec.3}

In this section, we prove Theorem~\ref{THM1}~\textbf{(I)}; that is, the  global solution $(u,w)$ to~\eqref{E} in the ball $\mathbb{B}_R$ is bounded for any initial value $(u^0,w^0)\in \mathcal{I}_{M,\bar{M},rad}(\mathbb{B}_R)$ with $M\in (0,M_c)$ and $\bar{M}\ge M$.

\medskip

Throughout this section, $c$ and $(c_i)_{i\ge 1}$ denote positive constants depending only on $d$, $R$, $M$, $\bar{M}$, and $(u^0,w^0)$. The dependence upon additional parameters is indicated explicitly.

\begin{proof}[Proof of Theorem~\ref{THM1}~\textbf{(I)}]
We consider $M\in (0,M_c)$, $\bar{M}\ge M$ and $(u^0,w^0)\in \mathcal{I}_{M,\bar{M},rad}(\mathbb{B}_R)$. Let $(u,w)$ be the corresponding global solution to~\eqref{E} and recall that $(u(t),w(t))\in \mathcal{I}_{M,\bar{M},rad}(\mathbb{B}_R)$ for all $t\ge 0$ according to Theorem~\ref{T1:Ex}. Then, by Proposition~\ref{PP2},
\begin{equation*}
	\mathcal{L}(u(t),w(t)) \le \mathcal{L}\big(u^0,w^0\big)\,, \quad t\ge 0\,,
\end{equation*}
which we combine with~\eqref{lbL} and the notation $v=\mathcal{K}[w-\langle w\rangle]$ to obtain, for $t\ge 0$, 
\begin{align*}
	\frac{1}{2} \left(1-\left(\frac{M}{M_c}\right)^{2/d}\right)\,\|\nabla v(t)\|_{2}^2 & \le \mathcal{L}\big(u^0,w^0\big) + c_0 \frac{M \bar{M}}{|\mathbb{B}_R|} + \frac{mM}{m-1}\,, \quad t\ge 0\,.
\end{align*}
Therefore, in view of $M<M_c$,
\begin{equation}\label{e20}
	\|\nabla v(t)\|_{2}\le c_1\,,\quad t\ge 0\,.
\end{equation}
As $d\ge 3$, we set $2^* := 2d/(d-2)>2$ and recall that $W_2^1(\mathbb{B}_R)$ is continuously embedded in $L_{2^*}(\mathbb{B}_R)$. By H\"older's inequality,
\begin{equation}
	\left| \int_{\mathbb{B}_R} u(t)v(t)\,\rd x\right| \le \|u(t)\|_{(2^*)'}\,\|v(t)\|_{2^*}\le c\|u(t)\|_{(2^*)'} \,\|\nabla v(t)\|_2\,,\quad t\ge 0\,.
\end{equation}
Writing
\begin{equation*}
	(2^*)'=\frac{2d}{d+2}=\vartheta+(1-\vartheta)m\,,\qquad \vartheta:=\frac{2}{d+2}\,,
\end{equation*}
we derive again with H\"older's inequality that
\begin{align*}
	\|u(t)\|_{(2^*)'}^{(2^*)'} & = \int_{\mathbb{B}_R} u(t)^{\vartheta+(1-\vartheta)m}\, \rd x \\
	& \le \|u(t)\|_1^\vartheta \|u(t)\|_m^{m(1-\vartheta)} = M^\vartheta \|u(t)\|_m^{m(1-\vartheta)}\,,
\end{align*}
hence
\begin{equation}\label{e14}
	\|u(t)\|_{(2^*)'}\le M^{1/d} \|u(t)\|_m^{m/2}\,,\quad t\ge 0\,.
\end{equation}
It then follows from~\eqref{e20}-\eqref{e14} that there is $c_2>0$ with
\begin{align}\label{e15}
\left| \int_{\mathbb{B}_R} u(t)v(t)\,\rd x \right| \le c_2 \|u(t)\|_m^{m/2}\,,\quad t\ge 0\,.
\end{align}
Invoking Proposition~\ref{PP2} and using~\eqref{eq.A2} and~\eqref{e15}, along with Young's inequality, we obtain that
\begin{align*}
	\mathcal{L}(u^0,w^0)&\ge \mathcal{L}(u(t),w(t))\ge \int_{\mathbb{B}_R} \phi\big(u(t)\big)\,\rd x - \left| \int_{\mathbb{B}_R} u(t) v(t)\,\rd x \right|\\
	&\ge \frac{1}{m-1}\|u(t)\|_m^m-\frac{m}{m-1}M - c_2 \|u(t)\|_m^{m/2} \\
	&\ge \frac{1}{2(m-1)} \|u(t)\|_m^m - \frac{mM}{m-1} - \frac{m-1}{2} c_2^2
\end{align*}
for $t\ge 0$. Consequently, $\|u(t)\|_m\le c_3$ for $t\ge 0$. Theorem~\ref{T1:Ex} now implies that 
\begin{equation*}
	\sup_{t\ge 0}\big(\|u(t)\|_\infty+\|w(t)\|_\infty\big)<\infty\,.
\end{equation*}
This yields Theorem~\ref{THM1}~\textbf{(I)}.
\end{proof}

\section{Proof of Theorem~\ref{THM1}~\textbf{(II)}}\label{sec.4}

Following the approach developed in \cite{HW2001} for the two dimensional case, the main step in the proof of Theorem~\ref{THM1}~\textbf{(II)} is to show that the energy functional $\mathcal{L}$ is bounded from below on the set of radially symmetric stationary solutions $(u,w)$ to~\eqref{E} with  $\|u\|_1=M>M_c$. Observing that $u=w$ for stationary solutions to~\eqref{E}, we define the set $\mathcal{S}_{M,rad}$ of radially symmetric stationary solutions to~\eqref{E} for $M>0$ as follows: $(u,w)\in\mathcal{S}_{M,rad}$ if and only if 
\begin{subequations}\label{Erad}
\begin{equation}
	u = w \in W_{p,rad}^{2,+}(\mathbb{B}_R)\,,\qquad \|u\|_1=M\,,
\end{equation}
and
\begin{align}
	\mathrm{div}\big(u\nabla(\phi'(u)-v)\big)&=0\ \text{ in }\ \mathbb{B}_R\,,\label{E1rad}\\
	\nabla u\cdot \textbf{n} & = 0 \ \text{ on }\ \partial\mathbb{B}_R \,,\label{NBCrad} \\
	v & = \mathcal{K}[u-\langle u\rangle]\ \text{ in }\ \mathbb{B}_R\,.\label{E2rad}
\end{align}
\end{subequations}
We point out that, according to the above definition, any $(u,w)\in \mathcal{S}_{M,rad}$ with $M>M_c$ belongs to $\mathcal{I}_{M,\bar{M},rad}(\mathbb{B}_R)$ for all $\bar{M}\ge M$.

\medskip

Next, we show that the analysis of $\mathcal{S}_{M,rad}$ can be reduced to that of radially symmetric and positive solutions to a single elliptic equation. To this end, given $a>0$, we denote by $\Sigma_a$ the set of functions 
\begin{equation*}
	z\in W_{p,rad}^{2}(\mathbb{B}_R)\,, \quad z(x)>0 \;\text{ for }\; x\in \mathbb{B}_R\,,\quad  \langle z\rangle = a\,,  
\end{equation*}
solving
\begin{equation*}
	-\Delta \phi'(z)=z-a\ \text{ in }\ \mathbb{B}_R\,,\qquad \nabla \phi'(z) \cdot \mathbf{n} = 0 \ \text{ on }\ \partial\mathbb{B}_R\,, 
\end{equation*}
and establish that $u\in \Sigma_{M/|\mathbb{B}_R|}$ for any $(u,w)\in \mathcal{S}_{M,rad}$. 

\begin{prop}\label{L19}
Let $M>0$ and $(u,w)\in\mathcal{S}_{M,rad}$. Then $u\in\Sigma_a$ for $a:=M/|\mathbb{B}_R|$. In fact, with $v = \mathcal{K}[u-\langle u\rangle]$ it holds that
\begin{equation}\label{n1}
	u(x)\ge\min\left\{1,\exp\left(m^{-1}\big(\langle \phi'(u)\rangle-\|v\|_\infty\big)\right)\right\}>0\,,\quad x\in\bar{\mathbb{B}}_R\,.
\end{equation}
Moreover, $\phi'(u)=v+\langle \phi'(u)\rangle$ satisfies
\begin{equation}\label{n3}
	\|\nabla\phi'(u)\|_2^2=\int_\Omega \phi'(u)\left(u-a\right)\,\rd x\,.
\end{equation}
\end{prop}

\begin{proof}
Note that $u$ is continuous on $\bar{\mathbb{B}}_R$ and let $P:=\{x\in\mathbb{B}_R\,:\, u(x)>0\}$ be the positivity set of~$u$. Then~\eqref{E1rad} and~\eqref{NBCrad} imply that $\nabla(\phi'(u)-v)=0$ in $P$. Hence, for each connected component $P_c$ of $P$, there is a constant $k_c\in\mathbb{R}$ such that 
\begin{equation}\label{n20}
	\phi'(u)=v+k_c\ \text{ in }\ P_c\,.
\end{equation}
Either $\min_{\bar{P}_c} u \ge 1$, from which we deduce that
\begin{equation*}
	P_c=P=\mathbb{B}_R \;\;\text{ and }\;\; u(x) \ge 1\,, \quad x\in\bar{\mathbb{B}}_R\,.
\end{equation*}
Or $\min_{\bar{P}_c} u < 1$ and we may pick any $x\in P_c$ with $0< u(x)<1$. Then~\eqref{phi} and~\eqref{n20} yield
\begin{equation*}
	-v(x)-k_c=-\phi'(u(x))=m\int_{u(x)}^1 \frac{(1+z)^{m-1}}{z}\,\rd z\ge -m\log(u(x))
\end{equation*}
and thus
\begin{equation*}
	u(x)\ge \exp\left(m^{-1}\big(k_c-\|v\|_\infty\big)\right)>0\,,
\end{equation*}
from which we again deduce that
\begin{equation*}
	P_c=P=\mathbb{B}_R \;\;\text{ and }\;\; u(x) \ge \exp\left(m^{-1}\big(k_c-\|v\|_\infty\big)\right) >0\,, \quad x\in\bar{\mathbb{B}}_R\,.
\end{equation*}
We have thus shown that $P_c=P=\mathbb{B}_R$, and we may integrate~\eqref{n20} over $\mathbb{B}_R$ and use $\langle v\rangle=0$ to conclude that $k_c=\langle\phi'(u)\rangle$. Gathering the outcome of the above analysis leads to~\eqref{n1} and the identity 
\begin{equation*}
    \phi'(u)-\langle \phi'(u)\rangle=v=\mathcal{K}\left[ u - \langle u\rangle\right]\,,
\end{equation*}
so that $\phi'(u)$ indeed solves
\begin{equation}\label{x0b}
	-\Delta \phi'(u)=u-a\ \text{ in }\ \mathbb{B}_R\,,\qquad \nabla \phi'(u) \cdot \mathbf{n} = 0 \ \text{ on }\ \partial\mathbb{B}_R\,,
\end{equation}
for $a=M/|\mathbb{B}_R|$ and satisfies~\eqref{n3}.
\end{proof}

\begin{cor}\label{C20}
Let $M>0$ and consider $(u,w)\in\mathcal{S}_{M,rad}$. Then
\begin{equation}\label{n4}
	\mathcal{L}(u,u) = \int_{\mathbb{B}_R} \phi(u)\,\rd x - \frac{1}{2} \|\nabla\phi'(u)\|_2^2
\end{equation}
and 
\begin{equation}\label{n5}
	M\langle \phi'(u)\rangle\le 2\mathcal{L}(u,u)+M\phi''(1)\,.
\end{equation}
\end{cor}

\begin{proof}
It follows from~\eqref{E2rad} that
\begin{equation*}
	\int_{\mathbb{B}_R} uv\,\rd x = \int_{\mathbb{B}_R} \big(u-\langle u\rangle\big) v\,\rd x =  - \int_{\mathbb{B}_R} v \Delta v\,\rd x = \|\nabla  v\|_2^2\,.
\end{equation*}
Therefore, since $\nabla v=\nabla\phi'(u)$ due to Proposition~\ref{L19}, we have
\begin{align*}
	\mathcal{L}(u,u) & = \int_{\mathbb{B}_R} \phi(u)\,\rd x - \int_{\mathbb{B}_R} uv\,\rd x + \frac{1}{2} \|\nabla v\|_2^2 \\
	& = \int_{\mathbb{B}_R} \phi(u)\,\rd x - \frac{1}{2} \|\nabla\phi'(u)\|_2^2\,,
\end{align*}
which is~\eqref{n4}. As for~\eqref{n5} we use~\eqref{n3}, \eqref{L20a} and~\eqref{n4} to derive
\begin{align*}
	M\langle \phi'(u)\rangle &=\int_{\mathbb{B}_R} u\phi'(u)\,\rd x - \|\nabla\phi'(u)\|_2^2\\
&\le m\int_{\mathbb{B}_R} \phi(u)\,\rd x + \phi''(1) \left( M-|\mathbb{B}_R| \right) - \|\nabla\phi'(u)\|_2^2\\
&=(m-2)\int_{\mathbb{B}_R} \phi(u)\,\rd x + \phi''(1) \left( M-|\mathbb{B}_R| \right) + 2 \mathcal{L}(u,u)\,.
\end{align*}
Since $\phi(u)\ge 0$, $\phi''(1)\ge 0$ and $m<2$, assertion~\eqref{n5} follows. 
\end{proof}


We next exploit further the radial symmetry of elements in $\Sigma_a$ to derive additional estimates.

\begin{lem}\label{Lx}
Let $a>0$, $M=a|\mathbb{B}_R|$ and consider $u\in \Sigma_a$. Then
\begin{equation}\label{n8}
	|\partial_r \phi'(u(r))| \le \frac{\|u\|_\infty}{d}r\,,\quad r\in (0,R]\,,
\end{equation}
and
\begin{equation}\label{n9}
	|\phi'(u(r))-\langle\phi'(u)\rangle| \le \frac{M}{\sigma_d(m-1)}\left(\frac{r^{2-d}}{d}+\frac{R^{2-d}}{2}\right)\,,\quad r\in (0,R]\,.
\end{equation}
Moreover,
\begin{equation}\label{n6}
	0<u(r)\le  \left( a R^d + \frac{a^2 R^d r^2}{d^2 m}\right) r^{-d}\,,\quad r\in (0,R]\,,
\end{equation}
and
\begin{align}\label{n7a}
	\partial_r u(r)\le  \frac{a}{dm} \cblue{\min\{u(r) , u(r)^{2-m} \}} r\,,\quad r\in (0,R]\,.
\end{align}
In fact,
\begin{equation}\label{n7}
	|\partial_r u(r)| \le \min\{u(r),u(r)^{2-m}\}\|u\|_\infty\frac{r}{dm}\,,\quad r\in (0,R]\,,
\end{equation}
and 
\begin{equation}\label{n7x}
	|\partial_r u(r)| \le \frac{M}{m\sigma_d}\min\{u(r),u(r)^{2-m}\}r^{1-d}\,,\quad r\in (0,R]\,.
\end{equation}
\end{lem}

\begin{proof}
Since $u$ solves~\eqref{x0b} and is radially symmetric,
\begin{equation*}
	-\partial_r\left(r^{d-1}\partial_r\phi'(u(r))\right)= r^{d-1} [u(r)-a]\,,\quad r\in (0,R]\,,
\end{equation*}
so that integration yields
\begin{align}\label{g1}
	-r^{d-1}\partial_r\phi'(u(r)) = \int_0^r s^{d-1} u(s)\,\rd s-\frac{a}{d}r^d \,,\quad r\in [0,R] \,.
\end{align}
On the one hand, since $u$ and $a$ are positive, we obtain from~\eqref{g1} that
\begin{align}\label{i9}
	-\frac{a}{d}r^d\le -r^{d-1}\partial_r\phi'(u(r))\le  \frac{ \|u\|_\infty }{d} r^d \,,\quad r\in [0,R] \,;
\end{align}
that is,
\begin{align*}
\vert \partial_r \phi'(u(r))\vert \le \frac{\|u\|_\infty}{d}r\,,\quad r\in [0,R] \,,
\end{align*}
since $a\le\|u\|_\infty$. This is~\eqref{n8}. On the other hand, we can use $\|u\|_1=M$ and the positivity of $u$ to derive from~\eqref{g1} that, for $r\in (0,R]$,
\begin{equation*}
	- \frac{M}{\sigma_d} = - \frac{aR^d}{d} \le - \frac{ar^d}{d} \le -r^{d-1}\partial_r\phi'(u(r)) \le \int_0^R s^{d-1} u(s)\,\rd s = \frac{M}{\sigma_d}\,,
\end{equation*}
hence
\begin{equation}\label{po}
	|\partial_r \phi'(u(r))| \le  \frac{M}{\sigma_d}r^{1-d}\,,\quad r\in (0,R]\,.
\end{equation} 
Next, we use~\eqref{po} to obtain, for $r\in (0,R]$,
\begin{align*}
	|\phi'(u(r))-\langle\phi'(u)\rangle| &= \left| \frac{d}{R^d}\int_0^R\big(\phi'(u(r))-\phi'(u(s))\big)\, s^{d-1}\,\rd s \right|\\
	&= \left| \frac{d}{R^d}\int_0^R\int_s^r \partial_r\phi'(u(s_*))\,\rd s_*\, s^{d-1}\,\rd s \right|\\
	&\le \frac{dM}{R^d\sigma_d} \int_0^R \left| \int_s^r s_*^{1-d}\,\rd s_*\right| \, s^{d-1}\,\rd s\\
	&\le  \frac{dM}{R^d\sigma_d(d-2)} \int_0^R\left(  r^{2-d}+s^{2-d} \right)\, s^{d-1}\,\rd s\\
	&= \frac{dM}{R^d\sigma_d(d-2)} \left(\frac{r^{2-d}R^d}{d}+\frac{R^2}{2}\right)\,,
\end{align*}
which is~\eqref{n9}. 

Noticing that
\begin{equation}\label{g2}
	\partial_r\phi'(u(r))=\phi''(u(r))\partial_r u(r)=m\frac{(1+u(r))^{m-1}}{u(r)}\partial_r u(r)\,,\quad r\in (0,R]\,,
\end{equation}
we derive from~\eqref{n8} that
\begin{equation*}
	|\partial_r u(r)|\le \frac{u(r)}{(1+u(r))^{m-1}}\|u\|_\infty\frac{r}{dm}\,,\quad r\in (0,R]\,,
\end{equation*}
and from~\eqref{po}  that
\begin{equation*}
	|\partial_r u(r)| \le \frac{u(r)}{(1+u(r))^{m-1}}\frac{M}{m\sigma_d}r^{1-d}\,,\quad r\in (0,R]\,.
\end{equation*}
Since $m>1$, we obtain~\eqref{n7} and~\eqref{n7x}. 

Finally, from~\eqref{i9}, \eqref{g2}, and the positivity of $u$ we infer that, for $r\in (0,R]$,
\begin{equation*}
	\partial_r u(r)\le \frac{a}{dm}\frac{u(r)}{(1+u(r))^{m-1}} r\le  \frac{a}{dm} \cblue{\min\{u(r),u(r)^{2-m}\}} r\,,\quad r\in (0,R]\,,
\end{equation*}
which is~\eqref{n7a}, and hence
\begin{equation*}
	\int_0^r s^{d}\partial_r u(s)\,\rd s\le  \frac{ar^2}{dm}\int_0^r s^{d-1} u(s)\,\rd s\le  \frac{a^2 R^d r^2}{d^2 m} \,,\quad r\in (0,R]\,,
\end{equation*}
 since $\|u\|_1=M=a|\cblue{\mathbb{B}_R}|$. Therefore,
\begin{equation*}
	r^{d}u(r) =\int_0^r\left(ds^{d-1}u(s)+s^{d}\partial_r u(s)\right)\, \rd s\le  a R^d + \frac{a^2 R^d r^2}{d^2 m} \,,\quad r\in (0,R]\,,
\end{equation*}
and hence~\eqref{n6}. This proves the lemma.
\end{proof}

The following result is at the heart of the analysis. Recall that $\zeta$ is defined in~\eqref{b0}.

\ccyan{
\begin{prop}\label{P20}
Let $a>0$, $M=a|\mathbb{B}_R|$ and consider a sequence $(u_l)_{l\ge 1}$ in $\Sigma_a$  such that $\lambda_l:=\|u_l\|_\infty^{1/d}\to \infty$ as $l\to\infty$. For $l\ge 1$, set 
\begin{equation}\label{t77}
	P_l(x):=\frac{1}{\lambda_l^d}u_l\left(\frac{x}{\lambda_l}\right)\,,\quad x\in \bar{\mathbb{B}}_{\lambda_l R}\,,
\end{equation} 
and
\begin{equation}\label{t7}
	\varphi_l(x):=\frac{1}{\lambda_l^{d-2}}\phi'\left(u_l\right)\left(\frac{x}{\lambda_l}\right)\,,\quad x\in \bar{\mathbb{B}}_{\lambda_l R}\,.
\end{equation}
Then there are $P\in C(\mathbb{R}^d)$ and $\varphi\in C^2(\mathbb{R}^d)$ such that, up to a subsequence, 
\begin{equation*}
	P_l\rightarrow P \ \text{ in }\ C(\mathbb{R}^d)\,,\qquad \varphi_l\rightarrow \varphi \ \text{ in }\ C^2(\mathbb{R}^d)\,.
\end{equation*}
Both $P$  and $\varphi$ are non-increasing with respect to $r=|x|$ with $P(0)=1$, and $P$ has compact support $\bar{\mathbb{B}}_{\rho_0}$ for some $\rho_0\in (0,\infty)$. Furthermore, $\varphi$ satisfies
\begin{equation}\label{varphi3}
	-\Delta \varphi=\left(\frac{m-1}{m}\right)^{1/(m-1)}\varphi_+^{1/(m-1)}\ \text{ in }\ \mathbb{R}^d
\end{equation}
with $\varphi_+=\max\{\varphi,0\}$, 
\begin{subequations}\label{420}
\begin{equation}
	P^{m-1} = \dfrac{m-1}{m} \varphi_+ \;\;\text{ in }\;\; \mathbb{R}^d\,, \label{420a}
\end{equation} 
and
\begin{equation}\label{420b}
 	\varphi(x)=\left\{\begin{array}{ll}
	\dfrac{m}{m-1} \rho_0^{2-d} \zeta\left(\dfrac{x}{\rho_0}\right)\,, & x\in \bar{\mathbb{B}}_{\rho_0}\,,\\[20pt]
	\dfrac{M_c}{\sigma_d}\dfrac{|x|^{2-d}-\rho_0^{2-d}}{d-2} \,, & x\in \R^d\setminus\bar{\mathbb{B}}_{\rho_0}\,.
\end{array}\right.
\end{equation}
\end{subequations}
Finally, it holds that $M\ge M_c$, the critical mass $M_c$ being defined in~\eqref{Mc}.
\end{prop}

\begin{proof}
Let $l\ge 1$. Since $u_l$ belongs to $\Sigma_a$, we infer from~\eqref{n1} and~\eqref{t77} that 
\begin{align}\label{n46}
	0< P_l(x)\le \|P_l\|_{L_\infty(\mathbb{B}_{\lambda_l R})} = 1\,,\quad x\in \bar{\mathbb{B}}_{\lambda_l R}\,,\quad \text{and}\quad \|P_l\|_{L_1(\mathbb{B}_{\lambda_l R})} = M\,.
\end{align}
We divide the proof into three steps.

\smallskip

\noindent\textbf{Step~1: Pointwise estimates on $(P_l)_{l\ge 1}$ and $(\varphi_l)_{l\ge 1}$.} 
It first follows from~\eqref{n7} that
\begin{equation}\label{n50}
 	|\partial_r P_l(r)| = \frac{1}{\lambda_l^{d+1}} \left| \partial_r u_l\left(\frac{r}{\lambda_l}\right)\right| \le \frac{1}{\lambda_l^{d+2}} \left[ u_l\left(\frac{r}{\lambda_l}\right) \right]^{2-m} \|u_l\|_\infty \frac{r}{dm}\le\frac{r}{dm} 
\end{equation}
for $r\in [0,\lambda_lR]$, since $\lambda_l=\|u_l\|_\infty^{1/d}$ and $m=2-2/d$.
Next, note from~\eqref{i9} that
\begin{equation}\label{t1}
	\partial_r\phi'(u_l)(r)\le \frac{ar}{d}  \,,\quad r\in [0,R] \,.
\end{equation}
Therefore, in view of~\eqref{phi},
\begin{align}
	\partial_r\big(\lambda_l^{-d}+P_l(r)\big)^{m} &= \frac{m}{\lambda_l^{d+1}} \left(\lambda_l^{-d}+\lambda_l^{-d}u_l\left(\frac{r}{\lambda_l}\right)\right)^{m-1} \partial_r u_l\left(\frac{r}{\lambda_l}\right)\nonumber\\
	&=\frac{1}{\lambda_l^{2d-1}} \phi''\left(u_l\right)\left(\frac{r}{\lambda_l}\right) u_l\left(\frac{r}{\lambda_l}\right) \partial_ru_l\left(\frac{r}{\lambda_l}\right)\nonumber\\
	&=\frac{1}{\lambda_l^{2d-1}} \partial_r\phi'\left(u_l\right) \left(\frac{r}{\lambda_l}\right) u_l\left(\frac{r}{\lambda_l}\right)\le \frac{ar}{d\lambda_l^d} \label{t2}
\end{align}
and
\begin{equation}\label{t5}
	\partial_r\varphi_l(r) = \frac{1}{\lambda_l^{d-1}} \partial_r\phi'(u_l) \left(\frac{r}{\lambda_l}\right) \le \frac{ar}{d\lambda_l^d}  
\end{equation}
for $r\in [0,\lambda_l R] $.
Integration of~\eqref{t2} and~\eqref{t5} with respect to $r$ now yields
\begin{equation}\label{Ps1}
 	\big(\lambda_l^{-d}+P_l(r_2)\big)^{m}- \big(\lambda_l^{-d}+P_l(r_1)\big)^{m}\le \frac{a(r_2^2-r_1^2)}{2d\lambda_l^d }  \,,\qquad 0\le r_1\le r_2\le \lambda_lR\,,
\end{equation}
and
\begin{equation}\label{Ps2}
	\varphi_l(r_2)-\varphi_l(r_1)\le \frac{a(r_2^2-r_1^2)}{2d\lambda_l^d }  \,,\qquad 0\le r_1\le r_2\le \lambda_lR\,,
\end{equation}
respectively. Moreover, it follows from~\eqref{Ps1} (with $r_1=0$ and $r_2=r$) that, for $r\in [0,\lambda_lR]$,
\begin{align*}
	\big(\lambda_l^{-d} + P_l(r)\big)^{m} \le \frac{a r^2}{2d\lambda_l^d } + \big(\lambda_l^{-d} + P_l(0)\big)^{m} \le \frac{a R^2}{2d\lambda_l^{d-2}} + \big(\lambda_l^{-d} + P_l(0)\big)^{m}\,,
\end{align*} 
which we combine with~\eqref{n46} to obtain
\begin{equation*}
    \big(\lambda_l^{-d}+1\big)^{m} = \big(\lambda_l^{-d} + \|P_l\|_{L_\infty(\mathbb{B}_{\lambda_l R})}\big)^{m}  \le \frac{a R^2}{2d\lambda_l^{d-2}} + \big(\lambda_l^{-d} + P_l(0)\big)^{m} \,.
\end{equation*}
The subadditivity of $z\mapsto z^{1/m}$ then entails that
\begin{equation*}
    \lambda_l^{-d}+1\le \lambda_l^{-d}+P_l(0)+\left(\frac{a R^2}{2d\lambda_l^{d-2}} \right)^{1/m}\,,
\end{equation*}
from which we deduce, using once more~\eqref{n46}, that
\begin{equation}\label{6.x}
    1\le P_l(0)+\left(\frac{a R^2}{2d\lambda_l^{d-2}} \right)^{1/m}\le 1+\left(\frac{a R^2}{2d\lambda_l^{d-2}} \right)^{1/m}\,.
\end{equation}
We next turn to the  connection between $P_l$ and $\varphi_l$. We use~\eqref{phi},~\eqref{t77}, and~\eqref{t7} to derive
\begin{align*}
	\varphi_l(r)& = \frac{1}{\lambda_l^{d-2}} \int_{1}^{u_l(r/\lambda_l)} \phi''(s)\,\rd s = \frac{m}{\lambda_l^{d-2}} \int_{\lambda_l^{-d}}^{P_l(r)} \frac{\left(1+\lambda_l^d s \right)^{m-1}}{s}\,\rd s
\end{align*}
for $r\in [0,\lambda_l R]$, hence
\begin{equation}\label{6.y}
    \varphi_l(r)=m \int_{\lambda_l^{-d}}^{P_l(r)} \frac{\left(\lambda_l^{-d}+s\right)^{m-1}}{s}\,\rd s\,,\quad r\in [0,\lambda_l R]\,.
\end{equation}
\cblue{According to~\eqref{6.x}, there is $l_0\ge 1$ such that $P_l(0)>1/2>\lambda_l^{-d}$ for $l\ge l_0$.} Combining~\eqref{6.x},~\eqref{6.y}, and the subadditivity of $z\mapsto z^{m-1}$ gives, \cblue{for $l\ge l_0$,}
\begin{align*}
    \varphi_l(0) & \ge m \int_{\lambda_l^{-d}}^{P_l(0)}s^{m-2} \,\rd s=\frac{m}{m-1}\big(P_l(0)-\lambda_l^{-d}\big)^{m-1}\\
    & \ge \frac{m}{m-1} \big(P_l(0)^{m-1}-\lambda_l^{-d(m-1)}\big)\,,
\end{align*}
hence
\begin{align}\label{6.zl}
\varphi_l(0)&\ge   \frac{m}{m-1} \cblue{\left[ \left(1-\left(\frac{a R^2}{2d\lambda_l^{d-2}} \right)^{1/m} \right)^{m-1} - \frac{1}{\lambda_l^{d-2}} \right]\,, \qquad l\ge l_0\,.}
\end{align}
Similarly, we infer from~\eqref{n46},~\eqref{6.y}, and the subadditivity of $z\mapsto z^{m-1}$ that, for $r\in [0,\lambda_l R]$,
\begin{align*}
\varphi_l(r)& \le m \int_{\lambda_l^{-d}}^{1} \left(\frac{\lambda_l^{-d(m-1)}}{s}+s^{m-2}\right) \,\rd s \\
& =\frac{m}{\lambda_l^{d-2}}\left(-\log\lambda_l^{-d}\right)+\frac{m}{m-1}\left(1-\lambda_l^{-d(m-1)}\right)
\end{align*}
and thus
\begin{equation}\label{6.zu}
    \varphi_l(r) \le \frac{m}{m-1}\left(1-\frac{1}{{\lambda_l^{d-2}}}\right)+\frac{md\log\lambda_l}{\lambda_l^{d-2}}\,,\quad r\in [0,\lambda_l R]\,.
\end{equation}
In particular, we deduce from~\eqref{6.zl} and~\eqref{6.zu} (with $r=0$) that there is $l_*>0$ such that
\begin{equation}\label{6.zz}
    \frac{m}{2(m-1)}\le \varphi_l(0)\le \frac{2m}{m-1}\,,\quad l\ge l_* \,.
\end{equation}
Furthermore, it follows from~\eqref{n8} for $r\in [0,\lambda_l R]$ that
\begin{align*}
    |\partial_r\varphi_l(r)|&=\frac{1}{\lambda_l^{d-1}}\left\vert\partial_r\phi'(u_l)\left(\frac{r}{\lambda_l}\right)\right\vert \le \frac{1}{\lambda_l^{d-1}} \frac{\|u_l\|_\infty}{d}\frac{r}{\lambda_l}=\frac{r}{d}\,,
\end{align*}
so that
\begin{align*}
    |\varphi_l(r)|&=\left|\varphi_l(0)+\int_0^r\partial_r\varphi(s)\,\rd s\right|\le \left|\varphi_l(0)\right|+\int_0^r\frac{s}{d}\,\rd s\le \left|\varphi_l(0)\right|+\frac{r^2}{2d}\,.
\end{align*}
Combining now this estimate with~\eqref{6.zz} yields
\begin{equation}\label{6.A}
    |\varphi_l(r)|\le\frac{2m}{m-1}+\frac{r^2}{2d}\,,\quad r\in [0,\lambda_l R]\,,\quad l\ge l_* \,.
\end{equation}

Finally, we note from~\eqref{x0b} that
\begin{equation}\label{t3}
	-\Delta \varphi_l = P_l - \frac{a}{\lambda_l^d}\ \text{ in }\ \mathbb{B}_{\lambda_l R}\,.
\end{equation}

\noindent\textbf{Step~2: Convergence of $(P_l)_{l\ge 1}$ and $(\varphi_l)_{l\ge 1}$.} Given $R_0>0$ and $l_0\ge l_*$ such that $\lambda_l R\ge 2R_0$ for $l\ge l_0$, we conclude from~\eqref{n46} and~\eqref{n50} that the sequence $(P_l)_{l\ge l_0}$ is bounded in $W_\infty^1(\mathbb{B}_{R_0})$. The Arzel\`a-Ascoli theorem then ensures the existence of a non-negative function $P\in C(\mathbb{R}^d)$ such that, up to a subsequence,
\begin{equation}\label{N51}
	P_l\to P \ \text { in }\ C(\mathbb{R}^d)
\end{equation}
as $l\to \infty$ (i.e. locally uniformly in $\mathbb{R}^d$).  

Similarly, given $R_0>0$, the right-hand side of the elliptic equation~\eqref{t3} is bounded in $W_\infty^1(\mathbb{B}_{2R_0})$ in view of~\eqref{n46}, \eqref{n50}, and the inequality
\begin{equation*}
	a = \frac{\|u_l\|_1}{|\mathbb{B}_R|} \le \|u_l\|_\infty = \lambda_l^d\,,
\end{equation*}
while $(\varphi_l)_{l\ge l_0}$ is bounded in $L_\infty(\mathbb{B}_{2R_0})$ according to~\eqref{6.A}. It thus follows from \cite[Theorems~9.11 and~9.19]{GT2001} that $(\varphi_l)_{l\ge l_0}$ is bounded in $W_q^3(\mathbb{B}_{R_0})$ for any $q\in (1,\infty)$. In particular, there exists $\varphi\in C^2(\mathbb{R}^d)$ such that, up to a subsequence, 
\begin{equation}\label{N51x}
 	\varphi_l\rightarrow \varphi \ \text{ in }\ C^2(\mathbb{R}^d)\,.
\end{equation}

\noindent\textbf{Step~3: Identifying $P$ and $\varphi$.}
Combining \eqref{t3},~\eqref{N51},  and~\eqref{N51x} implies that~$\varphi$ is a classical solution to 
\begin{equation}\label{t4}
	-\Delta \varphi=P \ \text{ in }\ \R^d\,.
\end{equation}
Next,  letting $l\to \infty$ in the inequalities \eqref{Ps1} and \eqref{Ps2} and using the convergences~\eqref{N51} and~\eqref{N51x} ensure that $P^m$ and $\varphi$, and hence also $P$, are non-increasing. Moreover, since $d\ge 3$, letting $l\to \infty$ in the inequalities~\eqref{6.zl} and \eqref{6.zu} (for $r=0$) with the help of~\eqref{N51} and~\eqref{N51x} gives $P(0)=1$ and $\varphi(0)=m/(m-1)$. In fact, for each $r>0$ with $P(r)>0$ we may let $l\to \infty$ in the identity~\eqref{6.y} to obtain
\begin{equation}\label{t8}
	\varphi(r) = m\int_{0}^{P(r)} s^{m-2}\,\rd s=\frac{m}{m-1} P^{m-1}(r)\,.
\end{equation}
To identify the positivity set of $P$ on which~\eqref{t8} is valid, we recall that $P$ is continuous and non-increasing on $[0,\infty)$ with $P(0)=1$, so that, either $P(r)>0$ for every $r\in [0,\infty)$, or there is $\rho_0\in (0,\infty)$ such that $\mathrm{supp}(P)=\bar{\mathbb{B}}_{\rho_0}$. We shall rule out the former. Assume for contradiction that $P>0$ on $\mathbb{R}^d$. Then~\eqref{t4} and~\eqref{t8} imply that $\varphi$ is a positive (classical) solution to
\begin{equation*}
	-\Delta \varphi=\left(\frac{m-1}{m}\right)^{1/(m-1)} \varphi^{1/(m-1)}\ \text{ in }\ \R^d\,,
\end{equation*}
which is impossible according to \cite[Theorem~1.1]{GS1981} since 
\begin{equation*}
    \frac{1}{m-1} = \frac{d}{d-2} < \frac{d+2}{d-2}\,.
\end{equation*}
Consequently, there is $\rho_0\in (0,\infty)$ such that $\mathrm{supp}(P)=\bar{\mathbb{B}}_{\rho_0}$, and~\eqref{t8} gives
\begin{equation*}
	P(x)=\left(\frac{m-1}{m}\right)^{1/(m-1)} \varphi^{1/(m-1)}(x)\,,\quad x\in \bar{\mathbb{B}}_{\rho_0}\,.
\end{equation*}
In addition, the monotonicity and continuity of $\varphi$, along with~\eqref{t8}, entail that
\begin{equation*}
	0\le \varphi_+(x)\le \varphi_+\left(\rho_0\frac{x}{|x|}\right)=0=\frac{m}{m-1}P(x)\,,\quad x\in \R^d\setminus\bar{\mathbb{B}}_{\rho_0}\,.
\end{equation*}
We infer from the previous two identities that $m P^{m-1}/(m-1) = \varphi_+$ in $\mathbb{R}^d$, an identity which we combine with~\eqref{t4} to obtain~\eqref{varphi3}.
Now, invoking~\cite[Theorem~1]{WY2003}, we deduce from~\eqref{b2x},~\eqref{b1}, and \eqref{varphi3} that $\varphi$ is given by~\eqref{420b}. Recalling~\eqref{t8}, we have thus established~\eqref{420}.

Finally, since $u_l\in\Sigma_a$ for all $l\ge 1$ with $M=a|\mathbb{B}_R|$, we deduce from~\eqref{b2x},~\eqref{420} and~\eqref{N51} that
\begin{align*}
	M & = \lim_{l\to\infty}\int_{\mathbb{B}_{\lambda_l R}}P_l(x)\,\rd x \ge \lim_{l\to\infty}\int_{\mathbb{B}_{\rho_0}} P_l(x)\,\rd x \\ 
	& = \int_{\mathbb{B}_{\rho_0}} P(x)\,\rd x = \frac{1}{\rho_0^{d}} \int_{\mathbb{B}_{\rho_0}} \zeta^{1/(m-1)}\left(\dfrac{x}{\rho_0}\right)\,\rd x = M_c\,,
\end{align*}
which completes the proof of Proposition~\ref{P20}.
\end{proof}}

In fact, we shall establish equality $M=M_c$ for \cblue{any} sequence of radially symmetric stationary states $(u_l,w_l)_{l\ge 1}$ in $\mathcal{S}_{M,rad}$ on which the Lyapunov functional $\mathcal{L}$ tends to $-\infty$. This will require rather precise local estimates in order to avoid that the rescaled sequence $(P_l)_{l\ge 1}$ ``loses mass at infinity''.

To this end, we first establish  a ``sup+inf''-type inequality based on Proposition~\ref{P20}. The arguments are adapted from \cite[Lemma~3]{LS1994} and \cite[Lemma~3]{WY2003}.

\begin{lem}\label{L25}
Let $a>0$ and 
\begin{equation}
	C_1\ge C_1^* := \dfrac{2\sigma_d(d-2)}{M_c} \left(\dfrac{m\zeta(0)}{m-1} + 1 \right)\,. \label{C1*} 
\end{equation}
Then there is $C_2>0$ such that, for any $u\in\Sigma_a$,
\begin{equation*}
	\phi'(u(0)) + C_1 \inf_{\mathbb{B}_\rho} \phi'(u) \le C_2\rho^{2-d}\,,\quad \rho\in (0,R]\,.
\end{equation*}
\end{lem}

\begin{proof}
Assume for contradiction that there exist $C_1\ge C_1^*$ and a sequence $(u_l,r_l)_{l\ge 1}$ in $\Sigma_a\times (0,R]$ such that
\begin{equation}\label{st2}
	\phi'(u_l(0))+C_1\inf_{\mathbb{B}_{r_l}} \phi'(u_l)\ge l r_l^{2-d}>0, \quad l\ge 1\,.
\end{equation}
Then, since $0\in \mathbb{B}_{r_l}$, $ r_l\in (0,R]$, and $d\ge 3$, it follows that
\begin{equation}\label{0}
	\lim_{l\to\infty} r_l^{d-2} \phi'(u_l(0)) = \lim_{l\to\infty} \phi'(u_l(0)) = \infty\,.
\end{equation}
Now, $\phi'$ being negative on $(0,1)$ according to~\eqref{phi}, we infer from~\eqref{L20b} and~\eqref{0} that $\lim\limits_{l\to\infty}u_l(0)=\infty$ and hence $\lambda_l^d:=\|u_l\|_\infty\to \infty$. In particular, there is $l_1\ge 1$ such that $1\le u_l(0)\le \lambda_l^d$ for $l\ge l_1$, and~\eqref{L20b} then implies that
\begin{equation*}
	\phi'(u_l(0)) \le \frac{m 2^{m-1}}{m-1} \lambda_l^{d(m-1)}\,, \quad l\ge l_1\,,
\end{equation*}
hence
\begin{equation*}
	r_l^{d-2}\phi'(u_l(0))\le \frac{m 2^{m-1}}{m-1} (r_l\lambda_l)^{d-2}\,, \quad l\ge l_1\,.
\end{equation*}
Invoking again~\eqref{0} yields 
\begin{equation}\label{0x}
	\lim_{l\to\infty} r_l\lambda_l=\infty\,.
\end{equation}
Moreover, introducing
\begin{equation*}
	\varphi_l(x):=\frac{1}{\lambda_l^{d-2}}\phi'\left(u_l\right)\left(\frac{x}{\lambda_l}\right)\,,\quad x\in \bar{\mathbb{B}}_{\lambda_l R}\,,
\end{equation*}
we deduce from Proposition~\ref{P20} that there exist $\rho_0\in (0,\infty)$ and a subsequence of $(\varphi_l)_{l\ge 1}$ (not relabeled) such that $\varphi_l\rightarrow \varphi \ \text{ in }\ C^2(\mathbb{R}^d)$ as $l\to\infty$, and $\varphi\in C^2(\mathbb{R}^d)$ is given by~\eqref{420}. In particular, for $x\in\mathbb{R}^d\setminus\mathbb{B}_{\rho_1}$ with $\rho_1>\rho_0$ yet to be determined, the representation formula~\eqref{420b} yields
\begin{align*}
	\varphi(0) + C_1 \varphi(x) & = \dfrac{m}{m-1} \dfrac{\zeta(0)}{\rho_0^{d-2}} + \dfrac{C_1 M_c}{\sigma_d (d-2)} \left( \frac{1}{|x|^{d-2}} - \frac{1}{\rho_0^{d-2}} \right)\\
	& \le \dfrac{m}{m-1} \dfrac{\zeta(0)}{\rho_0^{d-2}}  - \dfrac{C_1 M_c}{\sigma_d(d-2)} \frac{1}{2\rho_0^{d-2}} \\
	& \quad + \dfrac{C_1 M_c}{\sigma_d(d-2)} \left( \frac{1}{|x|^{d-2}} - \frac{1}{2\rho_0^{d-2}} \right) \\
	& \le \left( \dfrac{2m \zeta(0)}{m-1} - \dfrac{C_1 M_c}{\sigma_d(d-2)} \right) \frac{1}{2\rho_0^{d-2}} \\
	& \quad  + \dfrac{C_1 M_c}{\sigma_d(d-2)} \left( \frac{1}{\rho_1^{d-2}} - \frac{1}{2\rho_0^{d-2}} \right)\,.
\end{align*}
Therefore, recalling~\eqref{C1*} and taking $x\in\partial\mathbb{B}_{\rho_1}$ with $\rho_1:=2^{1/(d-2)}\rho_0>\rho_0$, we obtain
\begin{equation*}
	\varphi(0)+C_1 \inf_{\partial\mathbb{B}_{\rho_1}} \varphi \le -\frac{1}{\rho_0^{d-2}}\,.
\end{equation*}
Owing to~\eqref{0x} and the convergence of $(\varphi_l)_{l\ge 1}$ toward $\varphi$ in $C\left(\bar{\mathbb{B}}_{\rho_1}\right)$, we then find some $l_2\ge l_1$ such that $\min\{r_l,R\}>\rho_1/\lambda_l$ for $l\ge l_2$ and
\begin{equation*}
	\varphi_l(0) + C_1 \inf_{\partial\mathbb{B}_{\rho_1}} \varphi_l \le -\frac{1}{2\rho_0^{d-2}}\,, \quad l\ge l_2\,.
\end{equation*}
\cblue{Since $\partial\mathbb{B}_{\rho_1}\subset \bar{\mathbb{B}}_{\rho_1}$ and $\varphi_l$ is continuous, we further obtain
\begin{equation*}
	\varphi_l(0) + C_1 \inf_{\mathbb{B}_{\rho_1}} \varphi_l \le \varphi_l(0) + C_1 \inf_{\partial\mathbb{B}_{\rho_1}} \varphi_l \le - \frac{1}{2\rho_0^{d-2}}
\end{equation*}
for $l\ge l_2$. By} definition of $\varphi_l$, this entails that
\begin{equation}\label{r4}
	\phi'(u_l(0)) + C_1 \inf_{\mathbb{B}_{\rho_1/\lambda_l}} \phi'(u_l)  \le - \frac{1}{2} \left( \frac{\lambda_l}{\rho_0} \right)^{d-2}\,,\quad l\ge l_2\,.
\end{equation}
In particular, since $\mathbb{B}_{\rho_1/\lambda_l}\subset \mathbb{B}_{r_l}$ for $l\ge l_2$, we conclude from~\eqref{r4} that
\begin{equation*}
	\phi'(u_l(0)) + C_1 \inf_{\mathbb{B}_{r_l}} \phi'(u_l) \le - \frac{1}{2} \left( \frac{\lambda_l}{\rho_0} \right)^{d-2}\,,\quad l\ge l_2\,,
\end{equation*}
which contradicts~\eqref{st2}. This proves the claim.
\end{proof}

Based on the just derived \cblue{inequality, we} now establish a local estimate for functions in $\Sigma_a$, adapting an argument designed in \cite[Proposition~4]{WY2003}, see also \cite[Lemma~2]{LS1994}.

\begin{lem}\label{L28}
Let $a>0$. There is $\delta^*\in (0,1)$ depending only on $d$, $\zeta$, $a$, and $R$ with the following property: for any $\delta\in (0,\delta^*]$, there is $C_3(\delta)>0$ depending only on $d$, $\zeta$, $a$, $R$ and $\delta$ such that, for all $u\in \Sigma_a$,
\begin{equation*}
	\phi'(u(\rho)) \le C_3(\delta) \rho^{2-d} - \delta \phi'(u(0)) \,,\quad \rho\in (0,R/2]\,.
\end{equation*}
\end{lem}

\begin{proof}
Let $a>0$ and $u\in\Sigma_a$. \cblue{We fix $\rho\in (0,R/2]$ and observe} that $\phi'(u)-a|\cdot|^2/2d$ is a superharmonic function in~$\mathbb{B}_\rho$ due to~\eqref{x0b} and the non-negativity of $u$, since
\begin{equation*}
	-\Delta\left( \phi'(u)-\frac{a|x|^2}{2d} \right) = -\Delta \phi'(u)+a = u\ge 0 \ \text{ in }\ \mathbb{B}_R\,,
\end{equation*}
so that~\cite[Theorem~2.3]{GT2001} ensures that
\begin{equation*}
	\inf_{\partial\mathbb{B}_\rho} \left( \phi'(u) - \frac{a|\cdot|^2}{2d} \right) = \inf_{\mathbb{B}_\rho} \left( \phi'(u)-\frac{a|\cdot|^2}{2d} \right)\,.
\end{equation*}
This equality implies that 	\cblue{
\begin{align*}
	\phi'(u)(\rho) = \inf_{\partial\mathbb{B}_\rho} \phi'(u) & = \inf_{\partial\mathbb{B}_\rho} \left( \phi'(u) - \frac{a|\cdot|^2}{2d} \right) + \frac{a\rho^2}{2d}\\
	& = \inf_{\mathbb{B}_\rho} \left( \phi'(u) - \frac{a|\cdot|^2}{2d} \right) + \frac{a\rho^2}{2d} \le \inf_{\mathbb{B}_\rho} \phi'(u) + \frac{a\rho^2}{2d}\,.
\end{align*}
At this point, we set $\delta^* := 1/C_1^*$, see~\eqref{C1*}, and $C_1=1/\delta$ with $\delta\in (0,\delta^*]$. Then $C_1\ge C_1^*$ and we infer from Lemma~\ref{L25} that there is $C_2>0$ (depending on $C_1$ and thus on $d$, $\zeta$, $a$, $R$, and $\delta$) such that
\begin{equation*}
	\phi'(u)(\rho) \le \frac{1}{C_1} \left( C_2\rho^{2-d}-\phi'(u(0)) \right) +  \frac{a\rho^2}{2d} \le \left( \frac{C_2}{C_1} + \frac{a R^d}{2d}\right) \rho^{2-d} - \delta \phi'(u(0))
\end{equation*}
for} $\rho\in (0,R/2]$, and the proof is complete.
\end{proof}

Thanks to Lemma~\ref{L28}, we can now prove that an unbounded sequence of functions in $\Sigma_a$ concentrates at the origin and identify the corresponding space scale.

\begin{cor}\label{XX}
Let $a>0$ and consider a sequence $(u_l)_{l\ge 1}$ in $\Sigma_a$ such that $\lambda_l^d:=\|u_l\|_\infty\to \infty$.
Then there exists a constant $\kappa>0$ depending on $d$, $a$, and $R$ such that
\begin{equation}\label{limphi}
	\lim_{l\to\infty} \left(\max_{r\in [\kappa/\lambda_l,R/2] }u_l(r)\right)=0\,.
\end{equation}
\end{cor}


\begin{proof}
Consider $l\ge 1$ and pick $r_l\in [0,R]$ such that $u_l(r_l) = \lambda_l^d = \|u_l\|_\infty$. Owing to~\eqref{n6}, we find a constant $c_6=c_6(a,R)>0$ such that $\lambda_l^d = u_l(r_l) \le c_6 r_l^{-d}$, and thus
\begin{equation}\label{h0}
	r_l\le \frac{c_6^{1/d}}{\lambda_l}\,,\quad l\ge 1\,.
\end{equation}
We then use~\eqref{n7a} and~\eqref{h0} to obtain
\begin{align*}
	\lambda_l^d - u_l(0) & = \int_0^{r_l} \partial_r u_l(r)\,\rd r \le \frac{a}{dm} \int_0^{r_l} r \cblue{u_l(r)^{2-m}}\,\rd r \\
	& \le \frac{a}{2dm} \|u_l\|_\infty^{\cblue{2/d}} r_l^2 \le \cblue{\frac{a c_6^{2/d}}{2dm}}\,.
\end{align*}
Therefore, since $\lambda_l\to\infty$,  there is $l_3\ge 1$ such that
\begin{equation*}
	\frac{\lambda_l^d}{2} \le \left( 1-\frac{ac_6^{2/d}}{2dm} \frac{1}{\lambda_l^{\cblue{d}}} \right) \lambda_l^d \le u_l(0)\,,\quad l\ge l_3\,.
\end{equation*}
It now follows from~\eqref{L20b} and the identity $d(m-1)=d-2$ that
\begin{equation*}
	\frac{m}{m-1}\left(2^{1-m}\lambda_l^{d-2}-1\right)\le \phi'(u_l(0)) \,,\quad l\ge l_3\,.
\end{equation*}
As $d\ge 3$ and $\lambda_l\to\infty$, there is $c_7>0$ such that, for a possibly larger choice of $l_3$, 
\begin{equation}\label{h3}
	2c_7 \lambda_l^{d-2}\le \phi'(u_l(0)) \,,\quad l\ge l_3\,.
\end{equation}
Combining~\eqref{h3} with Lemma~\ref{L28} (with $\delta=\delta^*$) ensures that, for $\rho\in (0,R/2]$ and $l\ge l_3$,
\begin{align*}
	\phi'(u_l)(\rho) & \le C_3(\delta^*) \rho^{2-d} - \delta^* \phi'(u_l(0)) \le C_3(\delta^*) \rho^{2-d} - 2 c_7 \delta^* \lambda_l^{d-2} \\
	& = - c_7 \delta^* \lambda_l^{d-2} + C_3(\delta^*) \lambda_l^{d-2} \left[ (\rho \lambda_l)^{2-d} - \frac{c_7 \delta^*}{C_3(\delta^*)} \right]\,.
\end{align*}
Setting $\kappa^{2-d} := c_7 \delta^*/C_3(\delta^*)$, we infer from the above inequality that there is $l_4\ge l_3$ such that, for $l\ge l_4$, we have $\kappa \lambda_l^{-1} \le R/2$ and
\begin{equation}\label{h4}
	\phi'(u_l)(\rho) \le - c_7 \delta^* \lambda_l^{d-2}, \quad \rho\in [\kappa/\lambda_l,R/2]\,.
\end{equation}
In particular, $u_l(\rho)<1$ for $\rho\in [\kappa/\lambda_l,R/2]$ and $l\ge l_4$ by~\eqref{phi}.

Finally, for $l\ge l_4$ and $r\in [\kappa/\lambda_l,R/2]$,
\begin{align*}
	-\phi'(u_l)(r) & = \int_{u_l(r)}^1 \phi''(s)\,\rd s = m \int_{u_l(r)}^1 \frac{(1+s)^{m-1}}{s}\,\rd s \\
	& \le m 2^{m-1} \big(-\log{(u_l(r))}\big)\,.
\end{align*}
Equivalently, thanks to~\eqref{h4}, 
\begin{equation*}
	u_l(r) \le \exp\left\{ - \frac{c_7 \delta^*}{m 2^{m-1}} \lambda_l^{d-2} \right\}\,,
\end{equation*}
hence
\begin{equation*}
	\max_{r\in [\kappa/\lambda_l,R/2] } u_l \le \exp\left\{ - \frac{c_7 \delta^*}{m 2^{m-1}} \lambda_l^{d-2} \right\}\,, \quad l\ge l_4\,.
\end{equation*}
As $d\ge 3$ and $\lambda_l\to\infty$, the claim~\eqref{limphi} readily follows from the above inequality.
\end{proof}

We are now in a position to show that the energy functional $\mathcal{L}$ is unbounded from below on the set of stationary solutions $\mathcal{S}_{M,rad}$ only in the case $M=M_c$ with $M_c$ given in~\eqref{b2x}. This provides the last piece of information required for the proof of Theorem~\ref{THM1}~\textbf{(II)}.

\begin{prop}\label{P3}
Let $M>0$ and suppose there is a sequence $(u_l,w_l)_{l\ge 1}$ in $\mathcal{S}_{M,rad}$ satisfying
\begin{equation}\label{n39}
	\lim_{l\to\infty} \mathcal{L}(u_l,w_l) = -\infty\,.
\end{equation} 
Then $M=M_c$.
\end{prop}


\begin{proof} 
According to Proposition~\ref{L19}, $u_l$ belongs to $\Sigma_a$ for all $l\ge 1$ with $a:=M/|\mathbb{B}_R|$. Since~\eqref{n6} and \eqref{n7x} guarantee that the sequence $(u_l)_{l\ge 1}$ is bounded in~$C^1([\ve,R])$ for each $\ve\in (0,R)$, the Arzel\`a-Ascoli theorem and a diagonal process imply the existence of $u_\infty\in C((0,R])$ and a subsequence of $(u_l)_{l\ge 1}$ (not relabeled) such that
\begin{equation}
	\lim_{l\to\infty} \|u_l - u_\infty\|_{C([\varepsilon,R])} = 0 \;\;\text{ for all}\;\; \varepsilon\in (0,R)\,. \label{xx}
\end{equation}
It next follows from~\eqref{n5} and \eqref{n39} that $\langle \phi'(u_l)\rangle\to -\infty$ as $l\to\infty$ and therefore
\begin{equation}\label{n41}
	\lim_{l\to\infty} \phi'(u_l(r)) = -\infty\,,\quad r\in (0,R]\,,
\end{equation}
according to~\eqref{n9}. In particular, for each $r\in (0,R]$, we deduce from~\eqref{phi} and~\eqref{n41} that $u_l(r)<1$ for $l$ large enough (possibly depending on $r$) with
\begin{equation*}
	- \phi'(u_l)(r) = m \int_{u_l(r)}^1 \frac{(1+s)^{m-1}}{s}\,\rd s \le -m 2^{m-1} \log{u_l(r)}\,. 
\end{equation*}
Letting $l\to\infty$ in the above inequality and using~\eqref{n41} lead us to
\begin{equation*}
	\lim_{l\to\infty} u_l(r) = 0\,, \quad r\in (0,R]\,.
\end{equation*}
Therefore, $u_\infty(r) = 0$ for $r\in (0,R]$ and, recalling~\eqref{xx}, we have shown that
\begin{equation}\label{n42}
	\lim_{l\to\infty} \|u_l\|_{C([\varepsilon,R])} = 0 \;\;\text{ for all}\;\; \varepsilon\in (0,R)\,.
\end{equation}
Now, for $\varepsilon\in (0,R)$ and $l\ge 1$,
\begin{equation*}
	M = \|u_l\|_1 \le |\mathbb{B}_\varepsilon| \|u_l\|_\infty + |\mathbb{B}_R\setminus \mathbb{B}_\varepsilon| \|u_l\|_{C([\varepsilon,R])} \,,
\end{equation*}
and thus
\begin{equation*}
	\frac{M}{|\mathbb{B}_\varepsilon| } \le \|u_l\|_\infty + \frac{|\mathbb{B}_R|}{|\mathbb{B}_\varepsilon|} \|u_l\|_{C([\varepsilon,R])} \,.
\end{equation*}
An immediate consequence of~\eqref{n42} and the above inequality is that
\begin{equation*}
	\frac{M}{|\mathbb{B}_\varepsilon|} \le \liminf_{l\to\infty} \|u_l\|_\infty \;\;\text{ for all}\;\; \varepsilon\in (0,R)\,.
\end{equation*}
Setting $\lambda_l^d:=\|u_l\|_\infty$ and letting $\varepsilon\to 0$ in the above inequality entail that
\begin{equation}
 	\lim_{l\to\infty} \|u_l\|_\infty = \lim_{l\to\infty} \lambda_l = \infty\,. \label{n42.5}
\end{equation}
In view of~\eqref{n42.5}, we may apply Corollary~\ref{XX} and deduce from~\eqref{limphi} and~\eqref{n42} that there is $\kappa>0$ such that
\begin{equation}\label{limphil}
	\lim_{l\to\infty} \left( \max_{\bar{\mathbb{B}}_R\setminus\mathbb{B}_{\kappa/\lambda_l}} u_l \right)=0\,.
\end{equation}
Now, recall from Proposition~\ref{P20} that the sequence $(P_l)_{l\ge 1}$, defined in~\eqref{t77}, converges in $C(\mathbb{R}^d)$ to $P$ given by~\eqref{420}. Since
\begin{align*}
	M & = \|u_l\|_1 = \int_{\mathbb{B}_{\lambda_l R}} P_l(x)\,\rd x = \int_{\mathbb{B}_\kappa} P_l(x) \,\rd x + \int_{\mathbb{B}_{\lambda_l R}\setminus\mathbb{B}_\kappa} P_l(x) \,\rd x \\
	& = \int_{\mathbb{B}_\kappa} P_l(x) \,\rd x + \int_{\mathbb{B}_R\setminus\mathbb{B}_{\kappa/\lambda_l}} u_l(x)\,\rd x \le \int_{\mathbb{B}_\kappa} P_l(x)\,\rd x +  \left( \max_{\bar{\mathbb{B}}_R\setminus\mathbb{B}_{\kappa/\lambda_l}} u_l \right) |\mathbb{B}_R|
\end{align*}
for $\lambda_l>\kappa/R$, we use the convergence of $(P_l)_{l\ge 1}$, along with~\eqref{limphil} and Proposition~\ref{P20}, to pass to the limit as $l\to\infty$ in the above identity and find that
\begin{align*}
	M & \le \int_{\mathbb{B}_\kappa} P(x) \,\rd x = \left( \frac{m-1}{m} \right)^{1/(m-1)} \int_{\mathbb{B}_\kappa} \varphi_+^{1/(m-1)}(x)\,\rd x\\
	&\le \int_{\mathbb{R}^d} \frac{1}{\rho_0^{(d-2)/(m-1)}} \zeta^{1/(m-1)}\left( \frac{x}{\rho_0} \right)\,\rd x = M_c\,.
\end{align*}
Having already shown in Proposition~\ref{P20} that necessarily $M\ge M_c$, the assertion follows.
\end{proof}

\begin{cor} \label{Cf}
If $0<M\not= M_c$, then
\begin{equation*}
	\mu_M^s := \inf_{(u,w)\in \mathcal{S}_{M,rad}} \mathcal{L}(u,w) > - \infty\,.
\end{equation*}
\end{cor}

\begin{proof}
If $\mu_M^s = -\infty$, then there is a sequence $(u_l,w_l)_{l\ge 1}$ in $\mathcal{S}_{M,rad}$ satisfying~\eqref{n39}. Proposition~\ref{P3} now implies that $M=M_c$.
\end{proof}

\begin{proof}[Proof of Theorem~\ref{THM1}~\textbf{(II)}]
Let $\bar{M}\ge M>M_c$ and consider $(u^0,w^0)\in \mathcal{I}_{M,\bar{M},rad}(\mathbb{B}_R)$ such that
\begin{equation}
    \mathcal{L}(u^0,w^0) < \mu_M^s\,, \label{f01}
\end{equation}
the existence of such initial values being guaranteed by Proposition~\ref{prop1} and Corollary~\ref{Cf}. Setting $(u,w)=\boldsymbol{\Psi}\big(u^0,w^0\big)$, assume for contradiction that $u\in L_\infty((0,\infty)\times \Omega)$. It then follows from Theorem~\ref{T1:Ex}, Proposition~\ref{PP2}, Lemma~\ref{Wp}, Lemma~\ref{lem.contL}, and LaSalle's invariance principle that the $\omega$-limit set $\omega(u^0,w^0)$ of $\boldsymbol{\Psi}\big(u^0,w^0\big)$ is non-empty, compact in $W_{p}^{1,+}(\Omega) \times L_\infty^+(\Omega)$, and included in $\mathcal{S}_{M,rad}$ with 
\begin{equation*}
    \mathcal{L}(u_\infty,w_\infty) = \lim_{t\to\infty} \mathcal{L}(u(t),w(t)) \le \mathcal{L}(u^0,w^0)\,, \quad (u_\infty,w_\infty)\in \omega(u^0,w^0)\,.
\end{equation*}
Recalling~\eqref{f01}, we have thus found $(u_\infty,w_\infty)\in \mathcal{S}_{M,rad}$ such that \mbox{$\mathcal{L}(u_\infty,w_\infty)<\mu_M^s$}, contradicting the definition of $\mu_M^s$. This completes the proof of Theorem~\ref{THM1}~\textbf{(II)}. 
\end{proof}

\begin{rem}
In~\eqref{O2}, it is stated that
\begin{equation}
    \inf_{\mathcal{S}_{M,rad}} \mathcal{L} = \inf_{\mathcal{I}_{M,\bar{M},rad}(\mathbb{B}_R)} \mathcal{L}\,. \tag{\ref{O2}}
\end{equation}
Indeed, assuming 
\begin{equation*}
    \inf_{\mathcal{S}_{M,rad}} \mathcal{L} > \inf_{\mathcal{I}_{M,\bar{M},rad}(\mathbb{B}_R)} \mathcal{L}
\end{equation*}
readily leads to a contradiction by using the same argument as above in the proof of Theorem~\ref{THM1}~\textbf{(II)}.
\end{rem}


\appendix
\section{Properties of \texorpdfstring{$\phi$}{}}\label{sec.apA}
We derive some properties of the function $\phi: (0,\infty)\rightarrow (0,\infty)$ defined by
\begin{equation}
	\phi''(z)=m\frac{(1+z)^{m-1}}{z}\,,\quad z>0\,,\qquad \phi(1)=\phi'(1)=0\,, \tag{\ref{phi}}
\end{equation}
and begin with upper and lower bounds of $\phi$ in terms of the related power function $s\mapsto s^m$.

\begin{lem}\label{lemA1}
For each $\varepsilon>0$, there is a positive constant $c_\phi(\varepsilon)>0$, depending only on $m$ and $\varepsilon$, such that 
\begin{equation}
    \phi(s) \le \frac{1+\varepsilon}{m-1} s^m + c_\phi(\varepsilon)\,, \quad s\ge 0\,. \label{eqA1}
\end{equation}
In addition,
\begin{equation}
    \phi(s)\ge \frac{s^m}{m-1} + 1 - \frac{m}{m-1} s\,, \quad s\ge 0\,. \label{eq.A2}
\end{equation}
\end{lem}

\begin{proof}
Thanks to the sublinearity of $s\mapsto s^{m-1}$, there holds $\phi''(s)\le m\big(1+s^{m-2}\big)/s$ for $s>0$, an inequality which we integrate to obtain
\begin{align*}
    \phi'(s) & \le m \log{s} + \frac{m}{m-1} \big(s^{m-1} - 1\big)\,, \quad s\ge 1\,, \\
    -\phi'(s) & \le - m \log{s} + \frac{m}{m-1} \big(1 - s^{m-1}\big)\,, \quad s\in (0,1)\,.
\end{align*}
Integrating once more leads us to
\begin{align*}
    \phi(s) & \le m \big(s\log{s} - s +1\big) + \frac{s^m-1}{m-1} - \frac{m}{m-1} (s-1) \\
    & = \frac{s^m}{m-1} + m s\log{s} - \frac{m^2}{m-1} (s-1) - \frac{1}{m-1}
\end{align*}
for $s\ge 0$. Now, consider $\varepsilon>0$. Since $m s\log{s}\le 0$ for $s\in (0,1)$ and since there is $c(\varepsilon)>0$ depending only on $m$ and $\varepsilon$ such that $m s\log{s} \le \varepsilon s^m + c(\varepsilon)$ for $s\ge 1$, we readily obtain~\eqref{eqA1}.

Next, the function $\xi$, defined as
\begin{equation*}
	\xi(s):=\phi(s)-\frac{s^{m}-1}{m-1}-\frac{m}{m-1}(1-s)\,,\quad s\ge 0\,,
\end{equation*}
is convex on $(0,\infty)$ and has a minimum at $z=1$ according to~\eqref{phi}. Therefore, $\xi(s)\ge \xi(1)=0$ for $s\ge 0$, hence~\eqref{eq.A2}.
\end{proof}

We next derive estimates involving $\phi$ and $\phi'$, which we use throughout the analysis of stationary solutions to~\eqref{E}.

\begin{lem}\label{L20}
The function $\phi$ obeys the inequalities
\begin{equation}\label{L20a}
	s\phi'(s)\le m\phi(s) +\phi''(1)(s-1)\,,\quad s> 0\,,
\end{equation}
and
\begin{equation}\label{L20b}
	\frac{m}{m-1}(s^{m-1}-1)\le \phi'(s)\le \frac{m 2^{m-1}}{m-1} (s^{m-1}-1)\,,\quad s\ge 1\,,
\end{equation}
with $\phi'(s)<\cblue{0}$ for $s\in (0,1)$.
\end{lem}

\begin{proof}
Setting $\psi(z):=m\phi(z)-z\phi'(z)$ for $z>0$, it readily follows from~\eqref{phi} that
\begin{equation*}
	\psi''(z)=m(m-1)\frac{(1+z)^{m-2}}{z}\ge 0\,,\quad z>0\,,\qquad \psi(1)=0\,,\qquad \psi'(1)=-\phi''(1)\,.
\end{equation*}
Hence $\psi'(z)\ge -\phi''(1)$ for $z\ge 1$ while $\psi'(z)\le -\phi''(1)$ for $z\le 1$. Integration then yields~\eqref{L20a}.

Moreover, we note that, for $z\ge 1$,
\begin{equation*}
	\phi'(z)=m\int_1^z \frac{(1+s)^{m-1}}{s}\,\rd s\le m 2^{m-1}\int_1^z s^{m-2}\,\rd s
\end{equation*}
while
\begin{equation*}
	\phi'(z)=m\int_1^z \frac{(1+s)^{m-1}}{s}\,\rd s\ge m\int_1^z s^{m-2}\,\rd s\,.
\end{equation*}
This gives~\eqref{L20b}. Clearly, $\phi'(z)<\cblue{0}$ for $z\in (0,1)$ by~\eqref{phi}.
\end{proof}

\section{Properties of \texorpdfstring{$\zeta$}{}}\label{sec.apB}

We provide here connections between various quantities involving the solution $\zeta$ to~\eqref{b1}. Recall that
\begin{equation}
    M_c = \int_{\mathbb{B}_1} \zeta^{1/(m-1)}\,\rd x \label{eq.B1}
\end{equation}
according to~\eqref{b2x}.

\begin{prop}\label{propB1}
The potential $E_d*\zeta^{1/(m-1)}$ of $\zeta^{1/(m-1)}$ is given by
\begin{equation}
    \big(E_d*\zeta^{1/(m-1)}\big)(x) = \frac{m}{m-1} \zeta(x) + \frac{2-m}{m-1} \frac{I_\zeta}{M_c}\,, \quad x\in\mathbb{B}_1\,, \label{eq.B2}
\end{equation}
and
\begin{equation}
   \big(E_d*\zeta^{1/(m-1)}\big)(x) = \frac{M_c}{(d-2)\sigma_d} \frac{1}{|x|^{d-2}}\,, \quad x\in\mathbb{R}^d\setminus\mathbb{B}_1\,, \label{eq.B3} 
\end{equation}
where
\begin{equation}
    I_\zeta := \int_{\mathbb{B}_1} \zeta^{m/(m-1)}(x)\,\rd x\,. \label{eq.B4}
\end{equation}
Moreover,
\begin{equation}
    I_\zeta = \frac{m-1}{2-m}\frac{M_c^2}{(d-2)\sigma_d} = \frac{m}{m-1} \int_{\mathbb{B}_1} |\nabla\zeta(x)|^2\,\rd x\,. \label{eq.B5}
\end{equation}
\end{prop}

\begin{proof}
According to \cite[Propositions~3.4 \&~3.5]{BCL2009},
\begin{equation*}
    \frac{1}{m-1} \int_{\mathbb{R}^d} \zeta^{m/(m-1)}(x)\,\rd x = \frac{1}{2} \int_{\mathbb{R}^d} \zeta^{1/(m-1)}(x) \big(E_d*\zeta^{1/(m-1)}\big)\,\rd x\,,
\end{equation*}
from which we deduce that
\begin{equation}
    I_\zeta = \frac{m-1}{2} \int_{\mathbb{B}_1} \zeta^{1/(m-1)}(x) \big(E_d*\zeta^{1/(m-1)}\big)(x)\,\rd x\,, \label{eq.B6}
\end{equation}
after taking into account that the support of $\zeta$ is exactly the ball $\bar{\mathbb{B}}_1$. 

We next take advantage of the radial symmetry of $\zeta$ to deduce from \cite[Theorem~9.7, formula~(5)]{LL2001} that, for $x\in\mathbb{R}^d$,
\begin{equation}
\begin{split}
    \big(E_d*\zeta^{1/(m-1)}\big)(x) & = \frac{1}{(d-2)|x|^{d-2}} \int_0^{|x|} \zeta^{1/(m-1)}(s) s^{d-1}\,\rd s \\
    & \quad + \frac{1}{d-2} \int_{|x|}^\infty s \zeta^{1/(m-1)}(s)\, \rd s\,. \label{eq.B7}
\end{split}
\end{equation}
In particular, for $x\in\partial\mathbb{B}_1$, 
\begin{equation*}
    \left( E_d*\zeta^{1/(m-1)} - \frac{m}{m-1} \zeta \right)(x) = \frac{1}{(d-2)} \int_0^{1} \zeta^{1/(m-1)}(s) s^{d-1}\,\rd s = \frac{M_c}{(d-2)\sigma_d}
\end{equation*}
in view of~\eqref{eq.B1}, while it follows from~\eqref{b1} that
\begin{equation*}
    -\Delta\left( E_d*\zeta^{1/(m-1)} - \frac{m}{m-1} \zeta \right) = 0 \;\text{ in }\; \mathbb{B}_1\,.
\end{equation*}
Consequently,
\begin{equation}
   E_d*\zeta^{1/(m-1)} = \frac{m}{m-1} \zeta + \frac{M_c}{(d-2)\sigma_d} \;\text{ in }\; \mathbb{B}_1\,. \label{eq.B8}
\end{equation}
We now multiply~\eqref{eq.B8} by $\zeta^{1/(m-1)}$ and integrate over \cblue{$\mathbb{B}_1$} to obtain
\begin{align}
    & \int_{\mathbb{B}_1} \zeta^{1/(m-1)}(x) \big(E_d*\zeta^{1/(m-1)}\big)(x)\,\rd x \nonumber\\
    & \qquad = \frac{m}{m-1} I_\zeta + \frac{M_c}{(d-2)\sigma_d} \int_{\mathbb{B}_1} \zeta^{1/(m-1)}(x)\,\rd x \nonumber\\
    & \qquad = \frac{m}{m-1} I_\zeta + \frac{M_c^2}{(d-2)\sigma_d}\,, \label{eq.B9}
\end{align}
thanks to~\eqref{eq.B1}. The first identity in~\eqref{eq.B5} is then an immediate consequence of~\eqref{eq.B6} and~\eqref{eq.B9}, which we combine with~\eqref{eq.B8} to obtain~\eqref{eq.B2}, while the second identity in~\eqref{eq.B5} is deduced from~\eqref{b1} after multiplication by $\zeta$ and integration over $\mathbb{B}_1$.  Finally,~\eqref{eq.B3} readily follows from~\eqref{eq.B7} in view of~\eqref{eq.B1}, since the support of $\zeta$ is exactly the ball $\bar{\mathbb{B}}_1$. 
\end{proof}

\bibliographystyle{siam}
\bibliography{Lit_Chem}

\begin{thebibliography}{10}

\bibitem{Am1989}
{\sc H.~Amann}, {\em Dynamic theory of quasilinear parabolic systems. {III}.
  {G}lobal existence}, Math. Z., 202 (1989), pp.~219--250.

\bibitem{AmE90}
\leavevmode\vrule height 2pt depth -1.6pt width 23pt, {\em Erratum: ``{D}ynamic
  theory of quasilinear parabolic systems. {III}. {G}lobal existence'' [{M}ath.
  {Z}. {\bf 202} (1989), no. 2, 219--250]}, Math. Z., 205 (1990), p.~331.

\bibitem{Am1993}
\leavevmode\vrule height 2pt depth -1.6pt width 23pt, {\em {Nonhomogeneous
  linear and quasilinear elliptic and parabolic boundary value problems}}, in
  {Function spaces, differential operators and nonlinear analysis
  ({F}riedrichroda, 1992)}, vol.~133 of {Teubner-Texte Math.}, Teubner,
  Stuttgart, 1993, p.~9–126.

\bibitem{LQPP}
\leavevmode\vrule height 2pt depth -1.6pt width 23pt, {\em Linear and
  quasilinear parabolic problems. {V}ol. {I}}, vol.~89 of Monographs in
  Mathematics, Birkh\"{a}user Boston, Inc., Boston, MA, 1995.
\newblock Abstract linear theory.

\bibitem{Bi2020}
{\sc P.~Biler}, {\em Singularities of solutions to chemotaxis systems}, vol.~6
  of De Gruyter Series in Mathematics and Life Sciences, De Gruyter, Berlin,
  [2020] \copyright 2020.

\bibitem{BCL2009}
{\sc A.~Blanchet, J.~A. Carrillo, and {\relax Ph}.~Lauren\c{c}ot}, {\em
  Critical mass for a {P}atlak-{K}eller-{S}egel model with degenerate diffusion
  in higher dimensions}, Calc. Var. Partial Differential Equations, 35 (2009),
  pp.~133--168.

\bibitem{DLY2026}
{\sc C.~Dai, Y.~Li, and J.~Yan}, {\em Critical mass phenomenon in a nonlinear
  {K}eller-{S}egel system with indirect signal production}, J. Evol. Equ., 26
  (2026), pp.~Paper No. 61, 40.

\bibitem{FLT2023}
{\sc M.~Fuest, J.~Lankeit, and Y.~Tanaka}, {\em Critical mass phenomena in
  higher dimensional quasilinear {K}eller-{S}egel systems with indirect signal
  production}, Math. Methods Appl. Sci., 46 (2023), pp.~14362--14378.

\bibitem{FJ2022}
{\sc K.~Fujie and J.~Jiang}, {\em A note on construction of nonnegative initial
  data inducing unbounded solutions to some two-dimensional {Keller}-{Segel}
  systems}, Math. Eng. (Springfield), 4 (2022), p.~12.
\newblock Id/No 45.

\bibitem{GS1981}
{\sc B.~Gidas and J.~Spruck}, {\em Global and local behavior of positive
  solutions of nonlinear elliptic equations}, Comm. Pure Appl. Math., 34
  (1981), pp.~525--598.

\bibitem{GT2001}
{\sc D.~Gilbarg and N.~S. Trudinger}, {\em Elliptic partial differential
  equations of second order}, Classics in Mathematics, Springer-Verlag, Berlin,
  1998~ed., 2001.

\bibitem{Ho2002}
{\sc D.~Horstmann}, {\em On the existence of radially symmetric blow-up
  solutions for the {Keller}-{Segel} model}, J. Math. Biol., 44 (2002),
  pp.~463--478.

\bibitem{HW2001}
{\sc D.~Horstmann and G.~Wang}, {\em Blow-up in a chemotaxis model without
  symmetry assumptions}, European J. Appl. Math., 12 (2001), pp.~159--177.

\bibitem{JL2026}
{\sc T.~Jin and Y.~Li}, {\em Finite time blow-up in a quasilinear
  {K}eller-{S}egel system with indirect signal production}, Nonlinear Anal.
  Real World Appl., 89 (2026), pp.~Paper No. 104523, 15.

\bibitem{LSU1968}
{\sc O.~A. Lady\v{z}enskaja, V.~A. Solonnikov, and N.~N.
  Ural\textquotesingle~tseva}, {\em Linear and quasilinear equations of
  parabolic type}, vol.~Vol. 23 of Translations of Mathematical Monographs,
  American Mathematical Society, Providence, RI, 1968.
\newblock Translated from the Russian by S. Smith.

\bibitem{La2019}
{\sc {\relax Ph}.~Lauren\c{c}ot}, {\em Global bounded and unbounded solutions
  to a chemotaxis system with indirect signal production}, Discrete Contin.
  Dyn. Syst. Ser. B, 24 (2019), pp.~6419--6444.

\bibitem{LS1994}
{\sc Y.~Y. Li and I.~Shafrir}, {\em Blow-up analysis for solutions of {$-\Delta
  u=Ve^u$} in dimension two}, Indiana Univ. Math. J., 43 (1994),
  pp.~1255--1270.

\bibitem{LL2001}
{\sc E.~H. Lieb and M.~Loss}, {\em Analysis}, vol.~14 of Graduate Studies in
  Mathematics, American Mathematical Society, Providence, RI, second~ed., 2001.

\bibitem{ML2024}
{\sc X.~Mao and Y.~Li}, {\em Dirac-type aggregation with full mass in a
  chemotaxis model}, Discrete Contin. Dyn. Syst. Ser. S, 17 (2024),
  pp.~1513--1528.

\bibitem{TW2017}
{\sc Y.~Tao and M.~Winkler}, {\em Critical mass for infinite-time aggregation
  in a chemotaxis model with indirect signal production}, J. Eur. Math. Soc.
  (JEMS), 19 (2017), pp.~3641--3678.

\bibitem{WY2003}
{\sc G.~Wang and D.~Ye}, {\em On a nonlinear elliptic equation arising in a
  free boundary problem}, Math. Z., 244 (2003), pp.~531--548.

\end{thebibliography}

\end{document}